\documentclass[note,12pt]{amsart}
\usepackage{amscd,amsmath,amssymb, amsthm}
\theoremstyle{plain}
\newtheorem{thm}{Theorem}[section]
\theoremstyle{plain}
\newtheorem{lemma}[thm]{Lemma} 
\newtheorem{prop}[thm]{Proposition}

\newtheorem{remark}[thm]{Remark}

\newcommand\Ad{{\operatorname{Ad}}}

\newcommand\Span{{\operatorname{Span}}}
\newcommand\card{{\operatorname{card}}}
\newcommand\ind{{\operatorname{ind}}}

\newcommand\GCD{{\operatorname{GCD}}}
\newcommand\id{{\operatorname{id}}}

\begin{document}

\title{ The $K$-Theory of Toeplitz $C^*$-Algebras of Right-Angled Artin Groups } \par
\author{ Nikolay A. Ivanov }
\date{\today}

\address{\hskip-\parindent
Nikolay Ivanov  \\
Department of Mathematics \\
University of Toronto \\
Toronto, Ontario \\
Canada, M5S 2E4}
\email{nivanov@fields.utoronto.ca}

\begin{abstract}
Toeplitz $C^*$-algebras of right-angled Artin groups were studied by Crisp and Laca. 
They are a special case of the Toeplitz 
$C^*$-algebras $\mathcal{T}(G, P)$ associated with quasi-latice ordered groups $(G, P)$ introduced by Nica. 
Crisp and Laca proved that the so called 
"boundary quotients" $C^*_Q(\Gamma)$ of $C^*(\Gamma)$ are simple and purely infinite. 
For a certain class of 
finite graphs $\Gamma$ we show that $C^*_Q(\Gamma)$ can be represented as a full 
corner of a crossed product of an appropriate $C^*$-subalgebra of $C^*_Q(\Gamma)$ built by using $C^*(\Gamma')$,
where $\Gamma'$ is a subgraph of $\Gamma$ with one less vertex, by the group $\mathbb{Z}$. 
Using induction on the
number of the vertices of $\Gamma$ we show that $C^*_Q(\Gamma)$ are 
nuclear and belong to the small bootstrap class. 
We also use the Pimsner-Voiculescu exact sequence 
to find their $K$-theory. 
Finally we use the Kirchberg-Phillips classification theorem to show that those 
$C^*$-algebras are isomorphic to tensor products of $\mathcal{O}_n$ with $1 \leq n \leq \infty$. 
\end{abstract}

\maketitle

\section{Introduction}

Toeplitz $C^*$-Algebras of right-angled Artin Groups generalize both the Toeplitz algebra and the Cuntz algebras.
Coburn showed in \cite{C67} that the $C^*$-algebra, generated by a single nonunitaty isometry is unique, i.e.
every two $C^*$-algebras, each generated by a single nonunitary isometry are $*$-isomorphic. 
Similar uniqueness 
theorems about $C^*$-algebras generated by isometries were proved by Cuntz \cite{C77}, Douglas \cite{D72}, Murphy
\cite{M87}, and others. 
Laca and Raeburn in \cite{LR96} and Crisp and Laca in \cite{CL02} proved such uniquness 
theorems for a large class of $C^*$-algebras, corresponding to quasi-lattice ordered groups $(G, P)$. 
One of the
key point they use was to project onto the "diagonal" $C^*$-algebra generated by the range projections of those
isometries, an idea originating from \cite{D72}. 
\par
These $C^*$-algebras can be viewed as crossed products of commutative $C^*$-algebras (the $C^*$-algebras
generated by the range projections of the isometries) by semigroups of endomorphisms. 
Crisp and Laca used techniques
from \cite{ELQ02} about such crossed products together with the uniqueness theorems mentioned above to prove a 
sructure theorem for the universal $C^*$-algebra $C^*(G, P)$ (which by the uniqueness theorems is
isomorphic to the "reduced one" $\mathcal{T}(G, P)$) for a large class of quasi-lattice ordered groups $(G, P)$. 
We will now state 
\cite[Corollary 8.5]{CL07} and \cite[Theorem 6.7]{CL07} and use them throuought this note. A graph will always
mean a simple graph with countable set of vertices.

\begin{thm} [\cite{CL07}, Theorem 6.7] \label{thm:1}
Suppose that $\Gamma$ is a graph with a set of vertices $S$ (finite or infinite) such that $\Gamma^{\mathrm{opp}}$ 
has no isolated vertices. 
Then the universal $C^*$-algebra with generators $\{ V_s | s \in S \}$ subject to the relations: \\ 
(1) $V_s^* V_s = I$ for each $s \in S$; \\ 
(2) $V_s V_t = V_t V_s$ and $V_s^* V_t = V_t V_s^*$ if $s$ and $t$ are adjacent in $\Gamma$; \\ 
(3) $V_s^* V_t = 0$ if $s$ and $t$ are distinct and not adjacent in $\Gamma$; \\ 
(4) $\prod_{s \in S_{\lambda}} (I - V_s V_s^*) = 0$ for each $S_{\lambda} \subset S$ spanning a finite
connected component of $\Gamma^{\mathrm{opp}}$, \\
is purely infinite and simple.
\end{thm}

We will denote the $C^*$-algebra from this theorem by $C^*_Q(\Gamma)$. 

\begin{thm} [\cite{CL07}, Corollary 8.5] \label{thm:2}
Suppose that $\Gamma$ is a graph with a set of vertices $S$ (finite or infinite) such that 
$\Gamma^{\mathrm{opp}}$ has no
isolated vertices. 
Let $C^*(\Gamma)$ denote the universal $C^*$-algebra with generators $\{ V_s | s \in S \}$ 
subject to the relations: \\ 
(1) $V_s^* V_s = I$ for each $s \in S$; \\ 
(2) $V_s V_t = V_t V_s$ and $V_s^* V_t = V_t V_s^*$ if $s$ and $t$ are adjacent in $\Gamma$; \\ 
(3) $V_s^* V_t = 0$ if $s$ and $t$ are distinct and not adjacent in $\Gamma$; \\ 
Then each quotient of $C^*(\Gamma)$ is obtained by imposing a further collection of relations of the form \\ 
(R) $\prod_{s \in S_{\lambda}} (I - V_s V_s^*) = 0$, where each $S_{\lambda} \subset S$ spans a finite
union of finite connected components of $\Gamma^{\mathrm{opp}}$.
\end{thm}

We remind that by definition the opposite graph of the graph $\Gamma$ is 
$$\Gamma^{\mathrm{opp}} = \{ (v, w) | v, w \in S,\ (v, w) \notin \Gamma \}.$$
$\Gamma^{\mathrm{opp}}$ is also called the complement or the inverse of the graph $\Gamma$.
\par
Let $\Gamma$ be a finite graph with set of vertices $S$ such that the opposite graph $\Gamma^{\mathrm{opp}}$ is 
connected and has more than $1$ vertex. 
Then $C^*_Q(\Gamma)$ is the quotient of $C^*(\Gamma)$ by the ideal generated by $\underset{s \in
S}{\prod} (I - V_s V_s^*)$. 
Let $I_{\Gamma}: \langle \underset{s \in S}{\prod} (I - V_s V_s^*) \rangle_{C^*(\Gamma)}
\to C^*(\Gamma)$ be the inclusion map of this ideal, and $Q_{\Gamma}: C^*(\Gamma) 
\to C^*_Q(\Gamma)$ be the quotient map. 
Theorem \ref{thm:2} implicitly contains the uniqueness theorem
(\cite[Theorem 24]{CL02}). 
In particular we have the following faithful representation 
$\pi_{\Gamma}: C^*(\Gamma) \to \mathcal{B}(H_{\Gamma})$ 
which corresponds to $\mathcal{T}(A_{\Gamma}, A_{\Gamma}^+)$, where $A_{\Gamma} = \{ S | ss' = s's \text{ if }
(s,s') \in \Gamma \}$: \\ 
Let $H_{\Gamma}$ be the Hilbert space with an orthonormal basis 
$$\{ \mathfrak{E}[s_1, s_2, \dots, s_n] |\ n \in \mathbb{N}_0, s_1, \dots, s_n \in S \} / \sim,$$ 
where the relation $\sim$ means 
$\mathfrak{E}[s_1, s_2, \dots, s_n] \sim \mathfrak{E}[s'_1, s'_2, \dots, s'_m]$ if and only if $V_{s_1} 
\cdots V_{s_n} = V_{s'_1} \cdots V_{s'_m}$ subject to commutation relation (2) from Theorem \ref{thm:2}. \\ 
Let $\pi_{\Gamma}$ be given on a generating family of operators and vectors by 
$$\pi_{\Gamma}(V_s)(\mathfrak{E}[\emptyset]) = \mathfrak{E}[s],$$ 
$$\pi_{\Gamma}(V_s)(\mathfrak{E}[s_1, s_2, \dots, s_n]) = \mathfrak{E}[s, s_1, s_2, \dots, s_n].$$
For this representation it is true that the ideal 
$\langle \pi_{\Gamma}(\underset{s \in S}{\prod} (I - V_s V_s^*)) \rangle_{\pi_{\Gamma}(C^*(\Gamma))}$ 
coinsides with $\mathcal{K}(H_{\Gamma})$ - the compact operators on $H_{\Gamma}$. 
\par
In \cite{C77} Cuntz introduced a certain type of $C^*$-algebras $\mathcal{O}_n,\ n = 2,3, \dots, \infty$ 
generated 
by a set of isometries with mutually orthogonal ranges. 
He was able to represent $\mathcal{K} \otimes
\mathcal{O}_n$ as a crossed product of an $AF$-algebra by $\mathbb{Z}$ ($\mathcal{K}$ stands for the
$C^*$-algebra of the compact operators on a separable Hilbert space). 
There have been generalizations of these
algebras that depend on the "crossed product by $\mathbb{Z}$" idea, for example Cuntz-Krieger algebras \cite{CK80},
Cuntz-Pimsner algebras \cite{P97} and others. 
\par
In our note for a fixed finite graph with at least three vertices $\Gamma$ with $\Gamma^{\mathrm{opp}}$ connected 
we choose a subgraph $\Gamma'$ one less vertex such that $(\Gamma')^{\mathrm{opp}}$ is connected. 
Then we represent 
$C^*_Q(\Gamma)$ as a full corner of a crossed product of a $C^*$-algebra, built by using $C^*(\Gamma')$, by the
group $\mathbb{Z}$. 
After doing so we can use some results
about $C^*$-algebras which are crossed products by $\mathbb{Z}$. 
Most importantly we use the
Pimsener-Voiculescu exact sequence for the $K$-thoery (\cite{PV80}). 
Using induction on the number of the
vertices of the graph we conclude that $C^*_Q(\Gamma)$ is nuclear and belong to the small bootstrap class (see
\cite[IV.3.1]{B06}, \cite[\S22]{B98}) and thus the classification result for purely
infinite simple $C^*$-algebras of Kirchberg-Phillips \cite{Ph00} applies. 
From this we conclude that $C^*(\Gamma)$
is isomorphic to $\mathcal{O}_{1+|\chi(\Gamma)|}$, where $\chi(\Gamma)$ is an analogue of Euler
characteristic, introduced in \cite{CL07}. 
Then we extend this result to the case when $\Gamma$ is an infinite graph
with countably many vertices and such that $\Gamma^{\mathrm{opp}}$ is connected, since this graph can be 
represented as an
increasing sequence of finite subgraphs. 
The general case is a graph $\Gamma$ with at least two and at most
countably many vertices which is such that $\Gamma^{\mathrm{opp}}$ has no isolated vertex. 
It can be treated easily using Theorem \ref{thm:1} and the special cases described above. 
The conclusion is that $C^*_Q(\Gamma)$ is isomorphic to tensor
products of $\mathcal{O}_n$ for $1 \leq n \leq \infty$, where we define $\mathcal{O}_1$ to be the unital
Kirchberg algebra with ${\bf K}_0(\mathcal{O}_1) = \mathbb{Z} [1_{\mathcal{O}_1}]_0$ and 
${\bf K}_1(\mathcal{O}_1) = \mathbb{Z}$. 
A Kirchberg algebra is by definition a separable, nuclear, simple, purely infinite $C^*$-algebra that satisfies 
the Universal Coefficient Theorem. 

\section{Some $C^*$-Subalgebras of $C^*_Q(\Gamma)$ and the Crossed Product Construction}

If $\Gamma$ has two vertices and no edges, then from the construcion of $C^*(\Gamma)$ is clear that 
$C^*(\Gamma)$ is generated by isometries $V_1$ and $V_2$ with orthogonal ranges and such that 
$V_1 V_1^* + V_2 V_2^* < I$. 
This is the $C^*$-algebra $\mathcal{E}_2$ from \cite{C81} which is an extension of 
$\mathcal{O}_2$ by the compacts. Thus $C^*_Q(\Gamma) \cong \mathcal{O}_2$. 
\par
Suppose now that $\Gamma$ has a set of vertices $S$ such that $2 < \card(S) < \infty$ and suppose that 
$\Gamma^{\mathrm{opp}}$ is connected. 
Since $\Gamma^{\mathrm{opp}}$ is connected if it is not a tree we can remove an arbitrary edge from its
arbitrary cycle and the graph obtained in this way (let's denote it by $\Gamma^{\mathrm{opp}}_1$) will remain 
connected.
Continuing in this fashion in finitely many (say $l$) steps we will arrive at $\Gamma^{\mathrm{opp}}_l$ wich will 
be a tree. 
Let $s \in S$ be a "leaf" for $\Gamma^{\mathrm{opp}}_l$. 
Removing $s$ and the edge that comes out of $s$ from $\Gamma^{\mathrm{opp}}_l$ will not alter the connectedness. 
All this shows that if $\Gamma'$ is the graph, obtained from $\Gamma$ by removing the vertex $s$ and all the 
edges that come out of $s$, then its opposite graph $(\Gamma')^{\mathrm{opp}}$ will be connected.  
\par
Let $S' \subset S$ be the set of edges of $\Gamma'$. 
We can suppose that $S = \{ 1, \dots, n, n+1 \}$ and that $S' = \{ 1, \dots, n \}$ for some $n \geq 2$. 
We want to describe the words in letters 
$\{ V_1, \dots, V_n, V_{n+1}, V_1^*, \dots, V_n^*, V_{n+1}^* \}$.

\begin{lemma} \label{lemmma:words}
Every word in letters $\{ V_1, \dots, V_n, V_{n+1}, V_1^*, \dots, V_n^*, V_{n+1}^* \}$ can be written in the form 
$w_1 w_2^*$, where $w_1, w_2$ are words in letters $\{ V_1, \dots, V_n, V_{n+1} \}$.
\end{lemma}

\begin{proof}
We will use induction on the length of the words. 
The words of length one are $V_i$ and $V_i^*$ and they are of such form. 
Suppose that the statement of the lemma is true for all words of length $m > 1$ and less. 
Take a word $w$ of length $m+1$. 
We have two cases for $w$: \\ 
1) $w=w'V_i^*$ and 2) $w=w' V_i$ for some $1 \leq i \leq n+1$ and some word $w'$ of length $m$. \\ 
By the induction hypothesis $w'$ can be represented as $w' = w_1' (w_2')^*$, where $w_1'$ and $w_2'$ are words in 
letters $\{ V_1, \dots, V_n, V_{n+1} \}$. 
In case 1) $w=w_1'(w_2')^* V_i^*$, so setting $w_1 = w_1'$ and $w_2 = V_i w_2'$ shows that $w$ can be written in 
the desired form. 
For case 2) if the word $w_2'$ is empty then setting $w_1 = w_1' V_i$ and $w_2 = I$ shows that $w$ has the 
desired form. 
If $w_2' = V_j w_2''$ with $w_2''$ a word in letters $\{ V_1, \dots, V_{n+1} \}$ then 
$$w=w_1' (w_2'')^* V_j^* V_i=
      \begin{cases}
      0,& \text{ if } (i,j) \notin \Gamma \\ 
      w_1' (w_2'')^*,& \text{ if } i=j \\ 
      w_1' (w_2'')^* V_i V_j^*,& \text{ if } (i,j) \in \Gamma.
      \end{cases}$$
The first and the second case in the above equation are words of the desired form. 
In the third case we have that
$w_1' (w_2'')^* V_i$ is a word of length $m$ so it can be represented as $\omega_1 \omega_2^*$. 
Then $w_1' (w_2'')^* V_i V_j^* = \omega_1 \omega_2^* V_j^*$ is of the desired form. 
This concludes the induction and proves the lemma.
\end{proof}

Let's denote by $V$ the isometry $V_{n+1} \in C^*_Q(\Gamma)$ and suppose 
without loss of generality that $V^* V_i = 0$ for $k < i \leq n$ (notice that since $\Gamma^{\mathrm{opp}}$ is 
connected, $k < n$). 
If $k > 0$ then also $V$ commutes and $*$-commutes with $V_1, \dots, V_k$. 
\par
Let $T_0 = C^*(V_1, \dots, V_n)$. Then from Theorem \ref{thm:2} it is easy to see that $T_0 \cong C^*(\Gamma')$. 
Define by induction $T_m$ to be the closed linear span of elements of $C^*_Q(\Gamma)$ of the form 
$w V t_{m-1} V^* (w')^*$, where $w, w'$ are words in letters $\{ V_1, \dots, V_n \}$ and $t_{m-1} \in T_{m-1}$. 
The following lemma characterizes the sets $T_m$. 

\begin{lemma} \label{lemma:T}
$T_m$ is a $C^*$-subalgebra of $C^*(\Gamma)$, isomorphic to $\mathcal{K}^{\otimes m} \otimes T_0$ ($\cong
\mathcal{K} \otimes C^*(\Gamma')$). 
\end{lemma}

\begin{proof}
Let us denote by $\Omega$ the set of all words $\omega$ in letters $\{ V_1, \dots, V_n \}$ such that the letters
of the word $\omega V$ cannot be commuted pass $V$, i.e. $\omega V = \omega_1 V \omega_2$ for some 
words $\omega_1, \omega_2$ in letters $\{ V_1, \dots, V_n \}$, implies $\omega_2 = I$. 
It is easy to see that from the connectedness of $\Gamma^{opp}$ follows that $\Omega$ is an infinite countable 
set therefore we can enumerate its elements: $\Omega = \{ \omega_0, \omega_1, \omega_2, \dots \}$, setting 
$\omega_0 = I$. 
We assume that the words in $\Omega$ don't repeat, i.e. $\omega_p \neq \omega_q$ for $p \neq q$ after using the 
commutation relation. 
Suppose by induction that $T_{m-1} \cong \mathcal{K}^{\otimes (m-1)} \otimes T_0$ for some $m \geq 1$. 
We want to show that $T_m \cong \mathcal{K} \otimes T_{m-1}$. 
Clearly $\{ \omega_p V t_{m-1} V^* \omega_q^* | p,q \in \mathbb{N}_0 \}$ is a $*$-closed set. 
It is easy to see that each element $w' V t_{m-1} V^* w^*$ ot $T_m$ after applying the commutation relations (2) 
from Theorem \ref{thm:2} can be written in the form $\omega_p V t_{m-1}' V^* \omega_q^*$ for some 
$p,q \in \mathbb{N}_0$ and some $t_{m-1}' \in T_{m-1}$. 
Therefore $\{ \omega_p V t_{m-1} V^* \omega_q^* | p,q \in \mathbb{N}_0,\ t_{m-1} \in T_{m-1} \}$ spans a dense 
subset of $T_m$. 
We conclude that $T_m$ is $*$-closed. 
\par
We want to show now that $V^* \omega_q ^* \omega_p V = \delta_{p,q} I$. 
Write $\omega_p = V_{j_1} \cdots V_{j_s}$ and $\omega_q = V_{i_1} \cdots V_{i_t}$. 
Then $V^* \omega_q^* \omega_p V = V^* V_{i_t}^* \cdots V_{i_2}^* V_{i_1}^* V_{j_1} V_{j_2} \cdots V_{j_s} V$. 
There are three cases: \\ 
1) If $V_{j_1}$ commutes with $V_{i_1}^*, \dots, V_{i_r}^*$ ($1 \leq r < t$) and $i_{r+1} = j_1$ then 
$V_{i_{r+1}}^*$ will commute with $V_{i_1}^*, \dots, V_{i_r}^*$, so the word $\omega_q$ can be written in the 
form $\omega_q = V_{i_1} V_{i_2} \cdots V_{i_t}$ with $i_1 = j_1$. 
Then we can write $V^* \omega_q ^* \omega_p V = V^* V_{i_t}^* \cdots V_{i_2}^* V_{j_2} \cdots V_{j_s} V$ and 
continue the argument with this word.\\ 
2) If $V_{j_1}$ commutes with $V_{i_1}^*, \dots, V_{i_r}^*$ ($1 \leq r < t$) and $(j_1, i_{r+1}) \notin \Gamma$,
then $V^* \omega_q ^* \omega_p V = 0$. Also if $j_1 > k$ and $V_{j_1}$ commutes with $V_{i_1}^*, \dots,
V_{i_t}^*$ we also have $V^* \omega_q ^* \omega_p V = 0$.\\ 
3) If $V_{j_1}$ commutes with $V_{i_1}, \dots, V_{i_t}$ and $V$ clearly then $j_1 \leq k$ and from the definition 
of $\Omega$ follows that $V_{j_1}$ doesn't commute with all $V_{j_2}, \dots, V_{j_s}$. 
Suppose that $V_{j_1}$ doesn't commute with $V_{j_r}$ ($2 \leq r \leq s$) and if $r > 2$ $V_{j_1}$ commutes with 
$V_{j_2}, \dots, V_{j_{r-1}}$. 
Notice that $j_r \notin \{ i_1, \dots, i_t \}$ since $V_{j_1}$ commutes with $V_{i_1}^*, \dots, V_{i_t}^*$ and 
not with $V_{j_r}$. 
\par 
Suppose that $V^* \omega_q^* \omega_p V \neq 0$. 
Then suppose that $V_{j_1}, \dots, V_{j_{r_1}}$ can be dealt with by using repeatedly case 1). 
If $r_1 = s = t$ then $V^* \omega_q ^* \omega_p V = \delta_{p,q} I$ is proven. 
If $r_1 = s < t$ then $V^* \omega_q^* \omega_p V$ reduces to $V^* V_{i_t}^* \cdots V_{i_{s+1}}^* V$. 
If $i_t \leq k$ then $V_{i_t}^*$ would commute with $V^*$ contradicting the fact that $\omega_q \in \Omega$. 
$i_t > k$ implies immediatelly $V^* V_{i_t}^* \cdots V_{i_{s+1}}^* V = 0$ because $V$ does not commute with all 
of $V_{i_t}^*, \dots, V_{i_{s+1}}^*$ so it has a orthogonal range with some of them. 
The case $r_1=t<s$ is similar. 
If $r_1 < s$ and $r_1 < t$ then suppose that for $V_{j_{r_1+1}}$ case 3) applies. 
We will obtain a contradiction with the fact that $\omega_p \in \Omega$. 
By case 3) we can find $r_2 > r_1+1$ such that $V_{j_{r_1+1}}$ doesn't commute with
$V_{j_{r_2}}$ and if $r_2 > r_1+2$ then $V_{j_{r_1+1}}$ commutes with $V_{j_{r_1+2}}, \dots, V_{j_{r_2-1}}$. 
Also $j_{r_2} \notin \{ i_{r_1+1}, \dots, i_t \}$ ($V_{j_{r_1+1}}$ commutes with 
$V_{i_{r_1+1}}^*, \dots, V_{i_t}^*$ and not with $V_{j_{r_2}}$) and so case 1) cannot be applied to $V_{j_{r_2}}$. 
We can repeat this process finitely many times until we reach the isometry $V_{j_s}$ for which case 3) must apply 
since case 1) cannot be applied as we saw above and case 2) cannot be applied by assumption. 
But then $j_s \leq k$ and $V_{j_s}$ commutes with $V$ which contradicts $\omega_p \in \Omega$. 
This proves $V^* \omega_q ^* \omega_p V = \delta_{p,q} I$.
\par
It follows that $\omega_p V t_{m-1} V^* \omega_q^* \omega_{p'} V t_{m-1}' V^* \omega_{q'} = \delta_{p',q} 
\omega_p V t_{m-1} t_{m-1}' V^* \omega_q'$ and thus $T_m$ is a $C^*$-algebra. 
The equation $V^* \omega_q ^* \omega_p V = \delta_{p,q} I$ implies that 
$C^*( \{ \omega_p V V^* \omega_q^* | 0 \leq p,q \leq l-1 \} ) \cong M_l(\mathbb{C})$. 
It is clear that $V T_{m-1} V^*$ is a $C^*$-algebra, isomorphic to $T_{m-1}$. 
Therefore
\begin{multline*}
C^*(\{ \omega_p V t_{m-1} V^* \omega_q^* | 0 \leq p,q \leq l-1, t_{m-1} \in T_{m-1} \}) \cong \\ 
C^*(\{ \underset{i=0}{\overset{l-1}{\sum}} (\omega_i V t_{m-1} V^* \omega_i^*) | t_{m-1} \in T_{m-1} \}) 
\otimes C^*( \{ \omega_p V V^* \omega_q^* | 0 \leq p,q \leq l-1 \} ) \\ 
\cong T_{m-1} \otimes M_l(\mathbb{C}) = M_l(T_{m-1}), 
\end{multline*}
since $\underset{i=0}{\overset{l-1}{\sum}} (\omega_i V t_{m-1} V^* \omega_i^*)$ commutes with 
$\omega_p V V^* \omega_q^*$ for each $0 \leq p,q \leq l-1$ and each 
$t_{m-1} \in T_{m-1}$. 
Taking limit $l \to \infty$ concludes the proof of the lemma.
\end{proof}

From the proof of this lemma easily follows that $T_m$ is the closed linear span of 
$$\{ \omega_{p_m} V \cdots V \omega_{p_1} V t_0 V^* \omega_{q_1}^* V^* \cdots V^* \omega_{q_m}^* |\ \omega_{p_1}, 
\dots, \omega_{p_m}, \omega_{q_1}, \dots, \omega_{q_m} \in \Omega,\ t_0 \in T_0 \}.$$
This implies that $T_m \cdot T_l \subset T_m$ and $T_l \cdot T_m \subset T_m$ for each $m \geq l \geq 0$.  
\par
Now we introduce the following $C^*$-subalgebras of $C^*_Q(\Gamma)$: 
Define $B_0 = T_0$ and $B_m = C^*(B_{m-1} \cup T_m) = C^*(T_0 \cup \dots \cup T_m)$. 
From what we said above is clear that $T_m$ is an ideal of $B_m$. 
Therefore we have an extension 
\begin{equation} \label{equ:1}
0 \longrightarrow T_m \overset{i_m}{\longrightarrow} B_m \overset{p_m}{\longrightarrow} B_m/T_m \longrightarrow 
0,
\end{equation}
where $i_m : T_m \to B_m$ is the inclusion map and $p_m : B_m \to B_m/T_m$ is the quotient map.
\par
From \cite[Theorem 3.1.7]{M90} (or \cite[Corollary II.5.1.3]{B06}) follows that $B_m = B_{m-1} + T_m$ as a linear 
space. 
From \cite[Remark 3.1.3]{M90} follows that the map $\pi_m : B_{m-1}/(B_{m-1} \cap T_m) \to B_m/T_m$ given by 
$b_{m-1} + B_{m-1} \cap T_m \mapsto b_{m-1} + T_m$ is an isomorphism ($b_{m-1} \in B_{m-1}$). 
\par
Define $\mathcal{I}_m \overset{def}{=} \langle V^m [\underset{i=1}{\overset{n}{\prod}} (I - V_i V_i^*)] (V^*)^m
\rangle_{T_m}$. 
Since $T_0 \cong C^*(\Gamma')$ from Theorem \ref{thm:2} follows that $\mathcal{I}_0$ is the
unique nontrivial ideal of $T_0$ and it is isomorphic to $\mathcal{K}$. 
Then from Lemma \ref{lemma:T} follows
that $\mathcal{I}_m$ is the unique nontrivial ideal of $T_m$ and it is isomorphic to $\mathcal{K}^{\otimes m}
\otimes \mathcal{K}$. The ideal $\mathcal{I}_m$ can be described as the closed linear span of 
$$\{ \omega_{p_m} V \cdots V \omega_{p_1} V \iota_0 V^* \omega_{q_1}^* 
V^* \cdots V^* \omega_{q_m}^* |\ \omega_{p_1}, \dots, \omega_{p_m}, \omega_{q_1}, \dots, \omega_{q_m} \in \Omega,
\iota_0 \in \mathcal{I}_0 \}.$$
Therefore it is easy to see that $V^m (V^*)^m \mathcal{I}_m V^m (V^*)^m = V^m \mathcal{I}_0 (V^*)^m$. 
\par
By the definition of $C^*_Q(\Gamma)$ we have 
$(I - VV^*) \underset{i=1}{\overset{n}{\prod}} (I - V_i V_i^*) = 0$ or $\underset{i=1}{\overset{n}{\prod}} 
(I - V_i V_i^*) = VV^* \underset{i=1}{\overset{n}{\prod}} (I - V_i V_i^*)$. 
Therefore using relations (2) and (3) from Theorem \ref{thm:1} we get 
\begin{multline*}
\underset{i=1}{\overset{n}{\prod}} (I - V_i V_i^*) = VV^* \underset{i=1}{\overset{n}{\prod}} 
(I - V_i V_i^*) = V \underset{i=1}{\overset{k}{\prod}} (I - V_i V_i^*) V^* \underset{i=k+1}{\overset{n}{\prod}} 
(I - V_i V_i^*) = \\ 
= V \underset{i=1}{\overset{k}{\prod}} (I - V_i V_i^*) V^* \in T_1. 
\end{multline*}
It follows also that $V^m V \underset{i=1}{\overset{k}{\prod}} (I - V_i V_i^*) V^* (V^*)^m \in V^m T_1 (V^*)^m
\subset T_{m+1}$. 
It is easy to see that $T_{m+1} \cdot B_m \subset T_{m+1}$ and $B_m \cdot T_{m+1} \subset T_{m+1}$. 
This implies that $T_m \cap T_{m+1}$ is an ideal of $T_m$ and that $B_m 
\cap T_{m+1}$ is an ideal of $B_m$. 
From this we can conclude that $\mathcal{I}_m \subset (T_m \cap T_{m+1})$ for each $m \in \mathbb{N}$. 
The reverse inclusion is also true:

\begin{lemma} \label{lemma:3} 
$B_m \cap T_{m+1} = \mathcal{I}_m$ for each $m \in \mathbb{N}_0$.
\end{lemma}

\begin{proof}
Since $\mathcal{I}_0$ is the unique nontrivial ideal of $T_0$ and since $T_0 \cap T_1$ is an ideal of $T_0$, 
then if we assume that $\mathcal{I}_0 \subsetneq T_0 \cap T_1$ it will follow that 
$T_0 = T_0 \cap T_1$. 
Then $I = 1_{T_0} = 1_{C^*_Q(\Gamma)} \in T_0 \subset T_1$. 
This will imply that 
$T_1 \cong \mathcal{K} \otimes T_0$ is a unital $C^*$-algebra which is a contradiction. 
Therefore $\mathcal{I}_0 = T_0 \cap T_1$.
\par
It is easy to see that for each $m \in \mathbb{N}$ we have $V^m (V^*)^m T_m V^m (V^*)^m = V^m T_0 (V^*)^m \cong
T_0$ and that $V^m (V^*)^m T_{m+1} V^m (V^*)^m = V^m T_1 (V^*)^m \cong T_1$. 
Thus if we assume that $T_m = T_m \cap T_{m+1}$ it will follow that $V^m T_0 (V^*)^m \subset V^m T_1 (V^*)^m$ 
and therefore that $T_0 \subset T_1$. 
This is a contradiction with what we proved in the last paragraph. 
Therefore $T_m \cap T_{m+1} \subsetneq T_m$ and thus $T_m \cap T_{m+1} = \mathcal{I}_m$. 
\par
To conclude the proof of the lemma we have to show that $T_{m+1} \cap T_j = 0$ for each $0 \leq j < m$. 
In this case we have once again that $T_{m+1} \cap T_j$ is an ideal of $T_j$. 
Therefore the assumption $T_{m+1} \cap T_j \neq 0$ implies that $T_{m+1}$ contains the minimal nonzero ideal of 
$T_j$, $\mathcal{I}_j$. 
In particular $V^j
\underset{i=1}{\overset{n}{\prod}} (I - V_i V_i^*) (V^*)^j = V^{j+1} \underset{i=1}{\overset{k}{\prod}} (I - V_i
V_i^*) (V^*)^{j+1} \in T_{m+1}.$ 
This implies 
\begin{multline*} 
V^{j+1} \underset{i=1}{\overset{k}{\prod}} (I - V_i V_i^*) (V^*)^{j+1} = V^{j+1} (V^*)^{j+1} V^{j+1} 
\underset{i=1}{\overset{k}{\prod}} (I - V_i V_i^*) (V^*)^{j+1} 
V^{j+1} (V^*)^{j+1} \\ 
\in V^{j+1} (V^*)^{j+1} T_{m+1} V^{j+1} (V^*)^{j+1} = V^{j+1} T_{m-j} (V^*)^{j+1}. 
\end{multline*}
Therefore $\underset{i=1}{\overset{k}{\prod}} (I - V_i V_i^*) \in T_{m-j}$. 
Since also
$\underset{i=1}{\overset{k}{\prod}} (I - V_i V_i^*) \in T_0$, then the ideal $T_0 \cap T_{m-j}$ of $T_0$ contains
$\underset{i=1}{\overset{k}{\prod}} (I - V_i V_i^*)$. 
We will show that $\underset{i=1}{\overset{k}{\prod}} (I - V_i V_i^*) \notin \mathcal{I}_0$ this will imply that 
$T_0 \subset T_{m-j}$ for $m-j > 0$ and therefore obtaining a contradiction with the fact that $T_{m-j}$ is not 
unital for $m-j > 0$. 
\par
Suppose that $\underset{i=1}{\overset{k}{\prod}} (I - V_i V_i^*) \in \mathcal{I}_0$. Then since 
$T_0 = C^*(\Gamma')$ we have $Q_{\Gamma'}
(\underset{i=1}{\overset{k}{\prod}} (I - V_i V_i^*)) = 0$. 
From the connectedness of $(\Gamma')^{\mathrm{opp}}$ follows that we can find $j$, $1 \leq j \leq k$ and $l$, 
$k < l \leq n$ with $(j, l) \notin \Gamma'$. 
Then
\begin{multline*}
0 = Q_{\Gamma'}(V_l^*) Q_{\Gamma'}(\underset{i=1}{\overset{k}{\prod}} (I - V_i V_i^*)) Q_{\Gamma'}(V_l) = 
Q_{\Gamma'}(V_l^*) Q_{\Gamma'}(\underset{1 \leq i \leq k}{\underset{(i,l) \in \Gamma'}{\prod}} (I - V_i V_i^*))
Q_{\Gamma'}(V_l) = \\ 
= Q_{\Gamma'}(V_l^* V_l) Q_{\Gamma'}(\underset{1 \leq i \leq k}{\underset{(i,l) \in
\Gamma'}{\prod}} (I - V_i V_i^*)) = Q_{\Gamma'}(\underset{1 \leq i \leq k}{\underset{(i,l) \in
\Gamma'}{\prod}} (I - V_i V_i^*)). 
\end{multline*}
By repeating this argument finitely many times we will arrive at the equality $Q_{\Gamma'}(I) = 0$ which is a
contradiction. 
Therefore $\underset{i=1}{\overset{k}{\prod}} (I - V_i V_i^*) \notin \mathcal{I}_0$. 
This completes the proof of the lemma. 
\end{proof}

This lemma shows that we have an extension
\begin{equation} \label{equ:2}
0 \rightarrow \mathcal{I}_{m-1} \overset{i_m'}{\rightarrow} B_{m-1} \overset{p_m'}{\rightarrow}
B_{m-1}/ \mathcal{I}_{m-1} \rightarrow 0,
\end{equation}
where $i_m' : \mathcal{I}_{m-1} \to B_{m-1}$ is the inclusion map and $p_m' : B_{m-1} \to B_{m-1}/ 
\mathcal{I}_{m-1}$ is the quotient map.
\par
From equations (\ref{equ:1}) and (\ref{equ:2}) we have the commutative diagram with exact rows:
\begin{equation} \label{equ:3}
 \begin{CD}
    0 @>>> \mathcal{I}_{m-1} @>{i_m'}>> B_{m-1} @>{p_m'}>> B_{m-1}/\mathcal{I}_{m-1} @>>> 0\\
           &&    @V{I_m'}VV          @V{I_m}VV             @V{\cong}V{\pi_m}V      &&  \\
    0 @>>> T_m @>{i_m}>> B_m @>{p_m}>> B_m/T_m @>>> 0,
 \end{CD}
\end{equation}
where $I_m' : \mathcal{I}_{m-1} \to T_m$ and $I_m : B_{m-1} \to B_m$ are the inclusion maps. 
\par
Define $B \overset{def}{=} \overline{\underset{i=0}{\overset{\infty}{\bigcup}} B_i}^{\| . \|} \subset 
C^*(\Gamma)$ or in other words $B \overset{def}{=} \varinjlim (B_m, I_m)$. 
Notice that if $t_m \in T_m$ then $V t_m V^* \in T_{m+1}$. 
Thus we have a well defined injective endomorphism $\beta : B \to B$ given by $b \mapsto V b V^*$.
\par
Similarly to the Cuntz' construction from \cite{C77} we define $\tilde{B} \overset{def}{=} \varinjlim (B^m,
\alpha_m)$ as the limit of the sequence (which is also a commutative diagram) 
\begin{equation} \label{equ:4}
\begin{CD}
\dots @>{\alpha_{-m-1}}>> B^{-m} @>{\alpha_{-m}}>> \dots @>{\alpha_{-1}}>> B^0 @>{\alpha_0}>> B^1 @>{\alpha_1}>> 
\dots @>{\alpha_{m-1}}>> B^m @>{\alpha_{m}}>> \dots \\
& & @V{j_{-m}}V{\cong}V & & @V{j_0}V{\cong}V  @V{j_1}V{\cong}V & & @V{j_m}V{\cong}V  \\
\dots @>{\beta}>> B @>{\beta}>> \dots @>{\beta}>> B @>{\beta}>> B @>{\beta}>> \dots @>{\beta}>> B @>{\beta}>>
\dots, \\
\end{CD}
\end{equation} 
where $j_m : B^m \to B$ are $*$-isomorphisms. 
Since $\tilde{B}$ is a limit $C^*$-algebra we have $*$-homomorphisms $\alpha^m : B^m \rightarrow \tilde{B}$, 
s.t. $\alpha^m = \alpha^{m+1} \circ \alpha_m$ for all $m \in \mathbb{Z}$.
\par
Now we define a $*$-homorphism $\Phi$ of $\tilde{B}$ to itself, which is induced by "shift to the left" on
(\ref{equ:4}). 
In other words if we have a stabilizing sequence $(b^m)_{m= - \infty}^{ + \infty}$, where 
$b^m \in B^m$ for each $m$, then $\Phi((b^m)_{m= - \infty}^{ + \infty}) = (j_m^{-1} 
\circ j_{m+1}(b^{m+1}))_{m= - \infty}^{ + \infty}$. 
In particular for $b \in B$ the element $\alpha^m \circ j_m^{-1}(b)$ can be represented as the sequence 
$(0, \dots, 0, 0, j_m^{-1}(b), \alpha_m \circ j_m^{-1}(b), \alpha_{m+1} \circ \alpha_m \circ j_m^{-1}(b), 
\dots ) = (0, \dots, 0, 0, j_m^{-1}(b), j_{m+1}^{-1} \circ \beta(b), j_{m+2}^{-1} \circ \beta^2(b), \dots)$ 
therefore $\Phi(\alpha^m \circ j_m^{-1} (b))$ can be represented as the sequence $(0, \dots, 0, j_{m-1}^{-1}(b), 
j_{m}^{-1} \circ \beta(b), j_{m+1}^{-1} \circ \beta^2(b), \dots)$. 
This shows that $\Phi(\alpha^m \circ j_m^{-1} (b)) = \alpha^m \circ j_m^{-1} \circ \beta(b)$. 
The extension of this map to the whole of $\tilde{B}$ (we call it $\Phi$ also) is a $*$-isomorphism, because
$\Phi$ is isometric on the dense set of all stabilizing sequences (since $j_m$ are all isomorphisms). 
Now let $\tilde{A}$ be the crossed product of $\tilde{B}$ by the automorphism $\Phi$. 
We represent $\tilde{A}$ faithfully on a Hilbert space $\mathfrak{H}$ so that $\Phi$ is implemented by a unitary 
$U$ on $\mathfrak{H}$: $\Phi(b) = U b U^*$ for $b \in \tilde{B}$. Then $\tilde{A} = C^*(\tilde{B} \cup \{ U \})$. 
Every element of $\tilde{A}$ is a limit of elements of the form 
$\tilde{a} = \underset{i=-N}{\overset{N}{\sum}} b_iU^i = \underset{i=-N}{\overset{-1}{\sum}} U^i 
\bar{b_i} + b_0 + \underset{i=1}{\overset{N}{\sum}} b_i U^i$, with $b_i \in \tilde{B}$, where $\bar{b_i} = U^{-i} 
b_i U^i \in \tilde{B}$ for $i = -N, ..., -1$. 
Therefore the set of the elements of $\tilde{A}$ of the above form is dense in $\tilde{A}$. 
\par
Set $\tilde{P}_m \overset{def}{=} \alpha^m(1_{B^m}) \in \tilde{B}$ for each $m \in \mathbb{Z}$. 
Notice that $\alpha^m(1_{B^m}) = \alpha^m \circ j_m^{-1}(I) =  \alpha^{m+1} \circ \alpha_m \circ j_m^{-1}(I) = 
\alpha^{m+1} \circ j_{m+1}^{-1}(\beta(I))$. 
By induction 
$$\tilde{P}_m = \alpha^{m+i} \circ j_{m+i}^{-1}(\beta^i(I)),\ m \in \mathbb{Z},\ i \in \mathbb{N}.$$
Therefore we can write 
\begin{equation} \label{P}
\tilde{P}_m = \Phi^{-m}(\tilde{P}_0),\ m \in \mathbb{Z}.
\end{equation}

Consider the $C^*$-algebra $\tilde{P}_0 \tilde{A} \tilde{P}_0$. 
Clearly  $\tilde{P}_0 \tilde{B} \tilde{P}_0 \subset \tilde{P}_0 \tilde{A} \tilde{P}_0$. 
Since elements of the form $\tilde{a} = \underset{i=-N}{\overset{-1}{\sum}} U^i b_i + b_0 + 
\underset{i=1}{\overset{N}{\sum}} b_i U^i$ ($b_i \in \tilde{B}$) are dense in $\tilde{A}$, then elements of the 
form $\tilde{P}_0 \tilde{a} \tilde{P}_0 = \underset{i=-N}{\overset{-1}{\sum}} \tilde{P}_0 U^i b_i \tilde{P}_0 + 
\tilde{P}_0 b_0 \tilde{P}_0 + \underset{i=1}{\overset{N}{\sum}} 
\tilde{P}_0 b_i U^i \tilde{P}_0$ are dense in $\tilde{P}_0 \tilde{A} \tilde{P}_0$. 
It is easy to see that $U \tilde{P}_0 U^* = \Phi(\tilde{P}_0) < \tilde{P}_0$, so the range of $U\tilde{P}_0$ is 
contained in $\tilde{P}_0$ and therefore $\tilde{P}_0 U \tilde{P}_0 = U \tilde{P}_0$. 
Then 
\begin{multline*}
\tilde{P}_0 \tilde{a} \tilde{P}_0 =  \underset{i=-N}{\overset{-1}{\sum}} \tilde{P}_0 U^i b_i \tilde{P}_0 + 
\tilde{P}_0 b_0 \tilde{P}_0 + 
\underset{i=1}{\overset{N}{\sum}} \tilde{P}_0 b_i U^i \tilde{P}_0 = \\ 
 = \underset{i=-N}{\overset{-1}{\sum}} (\tilde{P}_0 U^i) (\tilde{P}_0 b_i \tilde{P}_0) + \tilde{P}_0 b_0 
\tilde{P}_0 + \underset{i=1}{\overset{N}{\sum}} (\tilde{P}_0 b_i \tilde{P}_0) (U^i \tilde{P}_0). 
\end{multline*}

This shows that if we set $S \overset{def}{=} U \tilde{P}_0$ then $\tilde{P}_0 \tilde{A} \tilde{P}_0 = 
C^*( \tilde{P}_0 \tilde{B} \tilde{P}_0 \cup \{ S \} )$. 
Let us also set $S_i \overset{def}{=} \alpha^0(j_0^{-1}(V_i)),\ i = 1, ..., n$. 
\par
It is easy to see that $\Span( \underset{l=0}{\overset{\infty}{\bigcup}} T_l )$ is dense in $B$. 
Then it follows that $\Span( \underset{i=0}{\overset{\infty}{\bigcup}} \alpha^i \circ j_i^{-1}
(\underset{l=0}{\overset{\infty}{\bigcup}} T_l) )$ is dense in $\tilde{B}$. 
Therefore $\tilde{P}_0 \Span( \underset{i=0}{\overset{\infty}{\bigcup}} \alpha^i \circ j_i^{-1}
(\underset{l=0}{\overset{\infty}{\bigcup}} T_l) ) \tilde{P}_0 = \Span( \tilde{P}_0 
\underset{i=0}{\overset{\infty}{\bigcup}} \alpha^i \circ j_i^{-1}
(\underset{l=0}{\overset{\infty}{\bigcup}} T_l) \tilde{P}_0 )$ is dense in $\tilde{P}_0  \tilde{B} \tilde{P}_0$.
For each $i \in \mathbb{N}$ we have
\begin{multline*} 
\tilde{P}_0 \alpha^i \circ j_i^{-1}(\underset{l=0}{\overset{\infty}{\bigcup}} T_l) \tilde{P}_0 = 
\alpha^i \circ j_i^{-1}(\beta^i(I)) \alpha^i \circ j_i^{-1}(\underset{l=0}{\overset{\infty}{\bigcup}} T_l)
\alpha^i \circ j_i^{-1}(\beta^i(I)) = \\ 
= \alpha^i \circ j_i^{-1}(\beta^i(I)
(\underset{l=0}{\overset{\infty}{\bigcup}} T_l) \beta^i(I)) = \alpha^i \circ j_i^{-1}(V^i (V^*)^i 
(\underset{l=0}{\overset{\infty}{\bigcup}} T_l) V^i (V^*)^i) = \\ 
= \alpha^i \circ j_i^{-1}((V^i (V^*)^i)^2 
(\underset{l=0}{\overset{\infty}{\bigcup}} T_l) (V^i (V^*)^i)^2) \subset 
\alpha^i \circ j_i^{-1}((V^i (V^*)^i T_i)
(\underset{l=0}{\overset{\infty}{\bigcup}} T_l) (T_i V^i (V^*)^i)) \subset \\ 
\subset \alpha^i \circ j_i^{-1}(V^i (V^*)^i 
(\underset{l=i}{\overset{\infty}{\bigcup}} T_l) V^i (V^*)^i) = \alpha^i \circ j_i^{-1}(V^i 
(\underset{l=0}{\overset{\infty}{\bigcup}}
T_l) (V^*)^i) = \alpha^i \circ j_i^{-1}(\beta^i(\underset{l=0}{\overset{\infty}{\bigcup}} T_l))) = \\ 
= \alpha^i \circ \alpha_{i-1} \circ \alpha_{i-2} \circ \dots \circ \alpha_1 \circ \alpha_0 \circ 
j_0^{-1}(\underset{l=0}{\overset{\infty}{\bigcup}} T_l) = \alpha^0 \circ
j_0^{-1}(\underset{l=0}{\overset{\infty}{\bigcup}} T_l). 
\end{multline*}
From this it follows that $\alpha_0 \circ j_0^{-1}(\Span(\underset{l=0}{\overset{\infty}{\bigcup}} T_l))$ is
dense in $\tilde{P}_0 \tilde{B} \tilde{P}_0$ and therefore also that $\alpha^0(B^0) = \tilde{P}_0 \tilde{B}
\tilde{P}_0$. 
This shows that $\tilde{P}_0 \tilde{A} \tilde{P}_0 = C^*(\alpha_0 \circ
j_0^{-1}(\underset{l=0}{\overset{\infty}{\bigcup}} T_l) \cup \{ S \})$. 
\par
Observe that 
\begin{multline} \label{equ:5}
S \alpha^0 \circ j_0^{-1}(b)S^* = U \tilde{P}_0 \alpha^0 \circ j_0^{-1}(b) \tilde{P}_0 U^* = U \alpha^0 \circ
j_0^{-1}(b) U^* = \Phi(\alpha^0 \circ j_0^{-1}(b)) = \\ 
= \alpha^0 \circ j_0^{-1} \circ \beta(b) = \alpha^0 \circ j_0^{-1}(VbV^*).
\end{multline}
Since for every $m > 0$ $T_m$ can be constructed from $T_0$ and $"\Ad(V)"$ equation (\ref{equ:5}) shows that 
$\tilde{P}_0 \tilde{A} \tilde{P}_0 = C^*(\alpha_0 \circ
j_0^{-1}(\underset{l=0}{\overset{\infty}{\bigcup}} T_l) \cup \{ S \}) = C^*(\alpha_0 \circ j_0^{-1}(T_0) \cup 
\{ S \}) = C^*( \{ S_1, \dots, S_n, S \}).$ 
\par
We want to apply now Theorem \ref{thm:1} to the $C^*$-algebra $A \overset{def}{=} \tilde{P}_0 \tilde{A}
\tilde{P}_0$. \\ 
$S_i = \alpha^0 \circ j_0^{-1} (V_i)$ are clearly isometries ($i=1, \dots, n$). 
$S^* S = \tilde{P}_0 U^* U \tilde{P}_0 = \tilde{P}_0$ and therefore $S$ is also an isometry. 
Thus condition (1) holds.\\ 
It is clear from (\ref{equ:5}) that $SS^* = \alpha^0 \circ j_0^{-1}(VV^*)$. 
Therefore 
\begin{multline*}
0 = \alpha^0 \circ j_0^{-1}(0) = 
\alpha^0 \circ j_0^{-1}((I - VV^*) \underset{i=1}{\overset{n}{\prod}} (I - V_i V_i^*)) = \\ 
= (\tilde{P}_0 - \alpha^0 \circ j_0^{-1}(VV^*))\underset{i=1}{\overset{n}{\prod}} (\tilde{P}_0 - 
\alpha^0 \circ j_0^{-1}(V_i V_i^*)) = 
(\tilde{P}_0 - SS^*) \underset{i=1}{\overset{n}{\prod}} (\tilde{P}_0 - S_i S_i^*).
\end{multline*}
This proves that condition (4) holds. \\ 
Conditions (2) and (3) obviously hold for all pairs of isometries from $\{ S_1, \dots, S_n \}$. 
If $n \geq i > k$ then $S_i S_i^* S S^* = \alpha^0 \circ j_0^{-1}(V_i V_i^* V V^*) = 0$, so condition (3) holds 
also for all paris $(S_i ,S)$ with $k < i \leq n$. 
For $1 \leq i \leq k$ one has 
\begin{multline*}
S S_i = S \alpha^0 \circ j_0^{-1}(V_i) = S \alpha^0 \circ j_0^{-1}(V_i) S^* S = \Phi(\alpha^0 \circ
j_0^{-1}(V_i)) S = \alpha^0 \circ j_0^{-1}(V V_i V^*) S = \\ 
= \alpha^0 \circ j_0^{-1}(V_i V V^*) S = \alpha^0 \circ
j_0^{-1}(V_i) \alpha^0 \circ j_0^{-1}(V V^*) S = S_i S S^* S = S_i S.
\end{multline*}
This shows that $S S_i = S_i S$. 
In the same way one can show that $S S_i^* = S_i^* S$. 
Therefore condition (4) holds for all pairs $(S, S_i)$ with $1 \leq i \leq k$. \\ 
Applying Theorem \ref{thm:1} we get $A \cong C^*_Q(\Gamma)$. 
Obviously we also have $C^*_Q(\Gamma) \cong \tilde{P}_m \tilde{A} \tilde{P}_m$ for each $m \in \mathbb{Z}$. 
\par
We reming here (see \cite[IV.3.1]{B06}, \cite[\S22]{B98}) that each $C^*$-algebra in the small bootstrap class 
$\mathfrak{N}$ satisfies the Universal Coefficient Theorem. 
The small bootstrap class $\mathfrak{N}$ is the smallest class of $C^*$-algebras that satisfy: \\ 
(i) $\mathbb{C} \in \mathfrak{N}.$ \\ 
(ii) $\mathfrak{N}$ is closed under stable isomorphism. \\ 
(iii) $\mathfrak{N}$ is closed under inductive limits. \\ 
(iv) $\mathfrak{N}$ is closed under crossed-products by $\mathbb{Z}$. \\ 
(v) If $0 \to \mathfrak{I} \to \mathfrak{A} \to \mathfrak{A}/\mathfrak{I} \to 0$ is an exact sequence, and two 
of $\mathfrak{I}, \mathfrak{A}, \mathfrak{A}/\mathfrak{I}$ are in $\mathfrak{N}$, so is the third.
\par
The $C^*$-algebras in this class are all nuclear.
\par
The following proposition holds:  

\begin{prop} \label{prop:4}
In the above settings: $\tilde{A} \cong \tilde{B} \rtimes_{\Phi} \mathbb{Z}$ and $A \cong C_Q^*(\Gamma)$ is 
Morita equivalent to $\tilde{A}$. 
Both of the $C^*$-algebras $A$ and $\tilde{A}$ are simple, belong to $\mathfrak{N}$ and 
${\bf K}_*(\tilde{A}) = {\bf K}_*(A)$. 
Also if we suppose that $[\tilde{P}_0]_0$ generates ${\bf K}_0(\tilde{A})$ then it follows that 
$[\tilde{P}_0]_0$ generates ${\bf K}_0(A)$. 
\end{prop}

\begin{proof}
We showed above that $\tilde{P}_m \tilde{A} \tilde{P}_m \cong C^*_Q(\Gamma)$ for each $m \in \mathbb{Z}$. 
It is easy to see that $\tilde{A} = \overline{\underset{m=0}{\overset{\infty}{\bigcup}} \tilde{P}_m \tilde{A}
\tilde{P}_m}$ and since each $\tilde{P}_m \tilde{A} \tilde{P}_m$ is simple from this follows that $\tilde{A}$ is 
simple too. 
Therefore every projection in $\tilde{A}$ is full. 
In particular $\tilde{P}_0$ is a full projection and therefore $A = \tilde{P}_0 \tilde{A} \tilde{P}_0$ is a full 
corner of $\tilde{A}$ and is therefore Morita equivalent to $\tilde{A}$. 
It follows that $A$ and $\tilde{A}$ are stably isomorphic (by Brown's Theorem \cite{B77}) and therefore 
${\bf K}_*(A) = {\bf K}_*(\tilde{A})$. 
\par 
If $\tilde{A}$ belongs to $\mathfrak{N}$ then from the definition follows that $A$ also does since it is stably 
isomorphic to $\tilde{A}$. 
\par
To conclude the proof of the lemma it remains to show that starting from any finite graph $G$ with
$G^{\mathrm{opp}}$ connected and going through the above construction the $C^*$-algebra (let us denote it by 
$\tilde{A}_G$ - the analogue of $\tilde{A}$ for $G$) belongs to $\mathfrak{N}$. 
We will do this by using induction on the number of the vertices of $G$. 
If $G$ has only two vertices and no edges then $C^*_Q(G) \cong \mathcal{O}_2$ and $C^*(G) \cong \mathcal{E}_2$ 
so the statement for this graph is true. 
Suppose that the statement is true for any graph $G$ with at most $n \geq 2$ vertices such that its opposite 
graph $G^{\mathrm{opp}}$ is connected. 
In particular $C^*_Q(\Gamma')$ (and therefore also $C^*(\Gamma')$) belong to $\mathfrak{N}$. 
Then $T_0 \cong C^*(\Gamma')$ as constructed above also does. 
Since the bootstrap category is closed under stabilization, extensions, inductive limits and crossed products by 
$\mathbb{Z}$ we conclude using induction that the $C^*$-algebra $\tilde{A}$ is also nuclear and belong to the 
small bootstrap class (we use diagram (\ref{equ:3}) together with Lemma \ref{lemma:T} and the fact that $\pi_m$ 
is an isomorphism for all $m \in \mathbb{N}$). 
Finally as we showed in the last paragraph this implies that $A$ belongs to $\mathfrak{N}$. 
This concludes the inductive step because $A \cong C^*_Q(\Gamma)$ and $\Gamma$ is an arbitrary graph with $n+1$ 
vertices such that $\Gamma^{\mathrm{opp}}$ is connected. 
\par
The final statement of the proposition in obvious. 
\par
The proposition is proved. 
\end{proof}

\section{The Computation of the ${\bf K}$-Theory} 

For a finite graph $G$ with $G^{\mathrm{opp}}$ connected Crisp and Laca conjectured in \cite{CL07} that the order 
of $[ 1_{C^*_Q(G)} ]_0$ in ${\bf K}_0(C^*_Q(G))$ is $| \chi(G) |$, where $\chi(G)$ is the Euler characteristics of 
$G$. 
$\chi(G)$ is defined as 
$$\chi(G) = 1 - \underset{j=1}{\overset{\infty}{\sum}} (-1)^{j-1} \times \{ \text{ number of complete subgraphs 
of $G$ on $j$ vertices } \}.$$

We will use the settings from the previous section. Denote $P_m \overset{def}{=} V^m (V^*)^m,\ m \in 
\mathbb{N}_0$. Denote also $Q \overset{def}{=} \underset{i=1}{\overset{k}{\prod}} (I - V_i V_i^*)$. 
Let $\Gamma_k = \{ (i,j) |\ 1 \leq i,\ j \leq k,\ (i,j) \in \Gamma' \} \subset \Gamma'$. 
\par
Since the vertex $n+1$ of $\Gamma$ is connected with each of the vertices $1, \dots, k$ and none of the others 
we have  
\begin{multline*}
\chi(\Gamma) = 1 - \underset{j=1}{\overset{n}{\sum}} (-1)^{j-1} \times \{ \text{ number of complete subgraphs 
of $\Gamma'$ on $j$ vertices } \} - \\ 
- ( 1 - \underset{j=1}{\overset{k}{\sum}} (-1)^{j-1} \times \{ 
\text{ number of complete subgraphs of $\Gamma_k$ on $j$ vertices } \} ).
\end{multline*}
Therefore 
\begin{equation} \label{Euler}
\chi(\Gamma) = \chi(\Gamma') - \chi(\Gamma_k).
\end{equation}

The following lemma is based on the "Euler characteristics idea" and is essentially due to Crisp and Laca:

\begin{lemma} \label{lemma:3.1}
If $E$ is a $C^*$-subalgebra of $B$ that contains $T_m$ (for $m \in \mathbb{N}_0$) we have 
\begin{equation} \label{equ:I}
\chi(\Gamma') [P_m]_0 = [P_{m+1} Q]_0 \text{ (in ${\bf K}_0(E)$).} 
\end{equation}

If $E$ is a $C^*$-subalgebra of $B$ that contains $T_m$ and $T_{m+1}$ (for $m \in \mathbb{N}_0$) we have 
\begin{equation} \label{equ:II}
\chi(\Gamma') [P_m]_0 = \chi(\Gamma_k) [P_{m+1}]_0 \text{ (in ${\bf K}_0(E)$).}
\end{equation}

If $E$ is a $C^*$-subalgebra of $B$ that contains $T_{m+1}$ (for $m \in \mathbb{N}_0$) we have 
\begin{equation} \label{equ:III}
[P_{m+1} Q]_0 = \chi(\Gamma_k) [P_{m+1}]_0 \text{ (in ${\bf K}_0(E)$).}
\end{equation}
\end{lemma}

\begin{proof}
In the last section we showed that
\begin{equation} \label{equ:6}
\underset{i=1}{\overset{n}{\prod}} (I - V_i V_i^*) = V \underset{i=1}{\overset{k}{\prod}} (I - V_i V_i^*) V^*.
\end{equation}
Since $V^m \underset{i=1}{\overset{n}{\prod}} (I - V_i V_i^*) (V^*)^m = \underset{i=1}{\overset{n}{\prod}} 
(V^m (V^*)^m - V^m V_i V_i^* (V^*)^m)$ then by multiplying equation (\ref{equ:6}) by $V^m$ on the left and by 
$(V^*)^m$ on the right we get 
\begin{multline*} 
\underset{i=1}{\overset{n}{\prod}} (V^m (V^*)^m - V^m V_i V_i^* (V^*)^m) = V^{m+1} 
\underset{i=1}{\overset{k}{\prod}} (I - V_i V_i^*) (V^*)^{m+1} = \\ 
= V^{m+1} (V^*)^{m+1} Q.
\end{multline*}

This equation is actually three equations which hold in certain $C^*$-subalgebras of $B$. 
We record them here:
\par
If $E$ is a $C^*$-subalgebra of $B$ that contains $T_m$ (for $m \in \mathbb{N}_0$) we have
\begin{equation} \label{equ:7.1}
\underset{i=1}{\overset{n}{\prod}} (V^m (V^*)^m - V^m V_i V_i^* (V^*)^m) = V^{m+1} (V^*)^{m+1} Q.
\end{equation}

If $E$ is a $C^*$-subalgebra of $B$ that contains $T_m$ and $T_{m+1}$ (for $m \in \mathbb{N}_0$) we have 
\begin{equation} \label{equ:7.2}
\underset{i=1}{\overset{n}{\prod}} (V^m (V^*)^m - V^m V_i V_i^* (V^*)^m) = V^{m+1} 
\underset{i=1}{\overset{k}{\prod}} (I - V_i V_i^*) (V^*)^{m+1}.
\end{equation}

If $E$ is a $C^*$-subalgebra of $B$ that contains $T_{m+1}$ (for $m \in \mathbb{N}_0$) we have 
\begin{equation} \label{equ:7.3}
V^{m+1} \underset{i=1}{\overset{k}{\prod}} (I - V_i V_i^*) (V^*)^{m+1} = V^{m+1} (V^*)^{m+1} Q.
\end{equation}

Note that if $E$ is an appropriate $C^*$-subalgebra of $B$ then for each projection 
$P$ that commutes with $V_1 V_1^*$ we have 
$[ V^m P (V^*)^m - V^m P V_1 V_1^* (V^*)^m ]_0 = [ V^m P (V^*)^m ]_0 - [ V^m P V_1 V_1^* (V^*)^m ]_0$. 
Suppose by induction that for some $n > l \geq 1$ if $P$ is a projection that commutes with 
$V_1 V_1^*, \dots, V_l V_l^*$ we have 
\begin{multline} \label{equ:P}
[ V^m \underset{i=1}{\overset{l}{\prod}}  (P - P V_i V_i^*) (V^*)^m ]_0 = \\ 
[ V^m P (V^*)^m ]_0 - \underset{i=1}{\overset{l}{\sum}} [ V^m P V_i V_i^* (V^*)^m ]_0 + \\ 
+ \underset{j=2}{\overset{l}{\sum}} 
(-1)^j (\underset{(i_s, i_t) \in \Gamma', 1 \leq s
< t \leq j}{\underset{ 1 \leq i_1 < \dots < i_j \leq l }{\sum}} [ V^m P V_{i_1} \cdots V_{i_j} 
V_{i_j}^* \cdots V_{i_1}^* (V^*)^m ]_0).
\end{multline}
We know that $V_{l+1} V_{l+1}^*$ commutes with each of $V_1 V_1^*, \dots, V_l V_l^*$. 
If $P$ commutes with $V_1 V_1^*, \dots, V_{l+1} V_{l+1}^*$ then we can apply (\ref{equ:P}) to the family 
$V_1 V_1^*, \dots, V_l V_l^*$ and the projection $P V_{l+1} V_{l+1}^*$ to obtain the following equation: 
\begin{multline*} 
[ V^m V_{l+1} V_{l+1}^* \underset{i=1}{\overset{l}{\prod}} (P - P V_i V_i^*) (V^*)^m ]_0 = 
[ V^m \underset{i=1}{\overset{l}{\prod}} (P V_{l+1} V_{l+1}^* - P V_{l+1} V_{l+1}^* V_i V_i^*) (V^*)^m ]_0 = \\
=[ V^m P V_{l+1} V_{l+1}^* (V^*)^m ]_0 - \underset{i=1}{\overset{l}{\sum}} 
[ V^m P V_{l+1} V_{l+1}^* V_i V_i^* (V^*)^m ]_0 + \\ 
+ \underset{j=2}{\overset{l}{\sum}} 
(-1)^j (\underset{(i_s, i_t) \in \Gamma', 1 \leq s
< t \leq j}{\underset{ 1 \leq i_1 < \dots < i_j \leq l }{\sum}} [ V^m P V_{l+1} V_{l+1}^* V_{i_1} 
\cdots V_{i_j} V_{i_j}^* \cdots V_{i_1}^* (V^*)^m ]_0).
\end{multline*}
Now since $V^m V_{l+1} V_{l+1}^* \underset{i=1}{\overset{l}{\prod}} (P - P V_i V_i^*) (V^*)^m < 
V^m \underset{i=1}{\overset{l}{\prod}} (P - P V_i V_i^*) (V^*)^m$ it is easy to see that we have  
\begin{multline*} 
[ V^m (P - P V_{l+1} V_{l+1}^*) \underset{i=1}{\overset{l}{\prod}} (P - P V_i V_i^*) (V^*)^m ]_0 = \\ 
= [ V^m \underset{i=1}{\overset{l}{\prod}} (P - P V_i V_i^*) (V^*)^m - V^m V_{l+1} V_{l+1}^* 
\underset{i=1}{\overset{l}{\prod}} (P - P V_i V_i^*) (V^*)^m ]_0 = \\ 
= [ V^m \underset{i=1}{\overset{l}{\prod}} (P - P V_i V_i^*) (V^*)^m ]_0 - 
[ V^m V_{l+1} V_{l+1}^* \underset{i=1}{\overset{l}{\prod}} (P - P V_i V_i^*) (V^*)^m ]_0 = \\ 
= [ V^m P (V^*)^m ]_0 - \underset{i=1}{\overset{l}{\sum}} [ V^m P V_i V_i^* (V^*)^m ]_0 + \\ 
+ \underset{j=2}{\overset{l}{\sum}} 
(-1)^j \underset{(i_s, i_t) \in \Gamma', 1 \leq s
< t \leq j}{\underset{ 1 \leq i_1 < \dots < i_j \leq l }{\sum}} [ V^m P V_{i_1} \cdots V_{i_j} 
V_{i_j}^* \cdots V_{i_1}^* (V^*)^m ]_0 - 
[ V^m P V_{l+1} V_{l+1}^* (V^*)^m ]_0 + \\ 
+ \underset{i=1}{\overset{l}{\sum}} [ V^m P V_{l+1} V_{l+1}^* V_i V_i^* (V^*)^m ]_0 - \\ 
- \underset{j=2}{\overset{l}{\sum}} 
(-1)^j \underset{(i_s, i_t) \in \Gamma', 1 \leq s
< t \leq j}{\underset{ 1 \leq i_1 < \dots < i_j \leq l }{\sum}} [ V^m P V_{l+1} V_{l+1}^* V_{i_1} 
\cdots V_{i_j} V_{i_j}^* \cdots V_{i_1}^* (V^*)^m ]_0 = \\ 
= [ V^m P (V^*)^m ]_0 - \underset{i=1}{\overset{l+1}{\sum}} [ V^m P V_i V_i^* (V^*)^m ]_0 + \\ 
+ \underset{j=2}{\overset{l+1}{\sum}} 
(-1)^j \underset{(i_s, i_t) \in \Gamma', 1 \leq s
< t \leq j}{\underset{ 1 \leq i_1 < \dots < i_j \leq l+1 }{\sum}} [ V^m P V_{i_1} \cdots V_{i_j} 
V_{i_j}^* \cdots V_{i_1}^* (V^*)^m ]_0.
\end{multline*}
Then by induction follows that for $l = k$ or $l = n$ we get  
\begin{multline*} 
[ \underset{i=1}{\overset{l}{\prod}} (V^m (V^*)^m - V^m V_i V_i^* (V^*)^m) ]_0 = \\ 
= [ I ]_0 - \underset{i=1}{\overset{l}{\sum}} [ V^m V_i V_i^* (V^*)^m ]_0 + 
\underset{j=2}{\overset{l}{\sum}} 
(-1)^j \underset{(i_s, i_t) \in \Gamma', 1 \leq s
< t \leq j}{\underset{ 1 \leq i_1 < \dots < i_j \leq l }{\sum}} [ V^m V_{i_1} \cdots V_{i_j} 
V_{i_j}^* \cdots V_{i_1}^* (V^*)^m]_0.
\end{multline*}

Combining the last equation with equations (\ref{equ:7.1}), (\ref{equ:7.2}) and (\ref{equ:7.3}) we obtain the 
following equations:
\par
If $E$ is a $C^*$-subalgebra of $B$ that contains $T_m$ (for $m \in \mathbb{N}_0$) we have
\begin{multline} \label{equ:8.1} 
[ V^m (V^*)^m ]_0 - \underset{i=1}{\overset{n}{\sum}} [ V^m V_i V_i^* (V^*)^m ]_0 + \\ 
+ \underset{j=2}{\overset{n}{\sum}} 
(-1)^j (\underset{(i_s, i_t) \in \Gamma', 1 \leq s
< t \leq j}{\underset{ 1 \leq i_1 < \dots < i_j \leq n }{\sum}} [ V^m V_{i_1} \cdots V_{i_j} 
V_{i_j}^* \cdots V_{i_1}^* (V^*)^m ]_0) = \\ 
= [ V^{m+1} (V^*)^{m+1} Q ]_0.
\end{multline}

If $E$ is a $C^*$-subalgebra of $B$ that contains $T_m$ and $T_{m+1}$ (for $m \in \mathbb{N}_0$) we have
\begin{multline} \label{equ:8.2} 
[ V^m (V^*)^m ]_0 - \underset{i=1}{\overset{n}{\sum}} [ V^m V_i V_i^* (V^*)^m ]_0 + \\ 
+ \underset{j=2}{\overset{n}{\sum}} 
(-1)^j (\underset{(i_s, i_t) \in \Gamma', 1 \leq s
< t \leq j}{\underset{ 1 \leq i_1 < \dots < i_j \leq n }{\sum}} [ V^m V_{i_1} \cdots V_{i_j} 
V_{i_j}^* \cdots V_{i_1}^* (V^*)^m ]_0) = \\ 
= [ V^{m+1} (V^*)^{m+1} ]_0 - \underset{i=1}{\overset{k}{\sum}} [ V^{m+1} V_i V_i^* (V^*)^{m+1} ]_0 + \\
+ \underset{j=2}{\overset{k}{\Sigma}} 
(-1)^j (\underset{(i_s, i_t) \in \Gamma_k, 1 \leq s
< t \leq j}{\underset{ 1 \leq i_1 < \dots < i_j \leq k }{\sum}} [ V^{m+1} V_{i_1} \cdots V_{i_j} 
V_{i_j}^* \cdots V_{i_1}^* (V^*)^{m+1} ]_0).
\end{multline}

If $E$ is a $C^*$-subalgebra of $B$ that contains $T_{m+1}$ (for $m \in \mathbb{N}_0$) we have
\begin{multline} \label{equ:8.3}
[ V^{m+1} (V^*)^{m+1} ]_0 - \underset{i=1}{\overset{k}{\sum}} [ V^{m+1} V_i V_i^* (V^*)^{m+1} ]_0 + \\
+ \underset{j=2}{\overset{k}{\sum}} 
(-1)^j (\underset{(i_s, i_t) \in \Gamma_k, 1 \leq s
< t \leq j}{\underset{ 1 \leq i_1 < \dots < i_j \leq k }{\sum}} [ V^{m+1} V_{i_1} \cdots V_{i_j} 
V_{i_j}^* \cdots V_{i_1}^* (V^*)^{m+1} ]_0) = \\ 
= [ V^{m+1} (V^*)^{m+1} Q ]_0.
\end{multline}

It is easy to see that in each $C^*$-subalgebra of $B$ that contains $T_m$ the projection $V^m V_{i_1} \cdots 
V_{i_j} V_{i_j}^* \cdots V_{i_1}^* (V^*)^m$ is Murray - von Neumann equivalent to $V^m (V^*)^m$ via the partial 
isometry $V^m V_{i_1} \cdots V_{i_j} (V^*)^m \in T_m$, where $\{ i_1, \dots, i_j \} \subset \{ 1, \dots, n \}$.
\\ 
This observation together with equations (\ref{equ:8.1}), (\ref{equ:8.2}) and (\ref{equ:8.3}) give:
\par
If $E$ is a $C^*$-subalgebra of $B$ that contains $T_m$ we have 
\begin{equation} \label{equ:9.1}
[P_m]_0 - \underset{i=1}{\overset{n}{\sum}} [P_m]_0 + \underset{j=2}{\overset{n}{\sum}} 
(-1)^j (\underset{(i_s, i_t) \in \Gamma', 1 \leq s
< t \leq j}{\underset{ 1 \leq i_1 < \dots < i_j \leq n }{\sum}} [P_m]_0) = [P_{m+1} Q]_0 .
\end{equation}

If $E$ is a $C^*$-subalgebra of $B$ that contains $T_m$ and $T_{m+1}$ then we have
\begin{multline} \label{equ:9.2}
[P_m]_0 - \underset{i=1}{\overset{n}{\sum}} [P_m]_0 + \underset{j=2}{\overset{n}{\sum}} 
(-1)^j (\underset{(i_s, i_t) \in \Gamma', 1 \leq s
< t \leq j}{\underset{ 1 \leq i_1 < \dots < i_j \leq n }{\sum}} [P_m]_0) = \\ 
= [P_{m+1}]_0 - \underset{i=1}{\overset{k}{\sum}} [P_{m+1}]_0 + \underset{j=2}{\overset{k}{\sum}} 
(-1)^j (\underset{(i_s, i_t) \in \Gamma_k, 1 \leq s
< t \leq j}{\underset{ 1 \leq i_1 < \dots < i_j \leq k }{\sum}} [P_{m+1}]_0).
\end{multline}

If $E$ is a $C^*$-subalgebra of $B$ that contains $T_{m+1}$ we have 
\begin{equation} \label{equ:9.3}
[P_{m+1}]_0 - \underset{i=1}{\overset{k}{\sum}} [P_{m+1}]_0 + \underset{j=2}{\overset{k}{\sum}} 
(-1)^j (\underset{(i_s, i_t) \in \Gamma_k, 1 \leq s
< t \leq j}{\underset{ 1 \leq i_1 < \dots < i_j \leq k }{\sum}} [P_{m+1}]_0) = [P_{m+1} Q]_0.
\end{equation}

The last three equations are what we had to prove.
\end{proof}

\begin{remark} \label{remark}
It also follows from this lemma that if we denote the isometries that generate $C^*(\Gamma)$ by $\tilde{V},
\tilde{V}_1, \dots, \tilde{V}_n$, then 
$$[ (I - \tilde{V} \tilde{V}^*) \underset{i=1}{\overset{n}{\prod}} (I - \tilde{V}_i \tilde{V}_i^*) ]_0 = 
\chi(\Gamma) [ I ]_0 \text{  (in $K_0(C^*(\Gamma))$).}$$ 
Therefore in the extenstion 
\begin{equation} \label{equ:remark}
0 \to \langle (I - \tilde{V} \tilde{V}^*) \underset{i=1}{\overset{n}{\prod}} (I - \tilde{V}_i \tilde{V}_i^*)
\rangle \overset{I_{\Gamma}}{\to} C^*(\Gamma) \overset{Q_{\Gamma}}{\to} C^*_Q(\Gamma) \to 0
\end{equation}
the map ${I_{\Gamma}}_*$ on ${\bf K}_0$ is given by 
$$[ (I - \tilde{V} \tilde{V}^*) \underset{i=1}{\overset{n}{\prod}} (I - \tilde{V}_i \tilde{V}_i^*) ]_0 \mapsto 
\chi(\Gamma) [ I ]_0.$$
\end{remark}

Now we can state and prove the following

\begin{prop} \label{prop:main}
Suppose that $G$ is a finite graph with at least two vertices and suppose that $G^{\mathrm{opp}}$ is connected. 
Then 
\begin{equation} \label{equ:K}
\begin{array}{cc}
{\bf K}_0(C^*_Q(G)) = 
                           \begin{cases}
			   \mathbb{Z}_{| \chi(G) |}, & \text{ if } \chi(G) \neq 0, \\ 
			   \mathbb{Z}, & \text{ if } \chi(G) = 0,
			   \end{cases}
			   
			   &
			   
\text{  } {\bf K}_1(C^*_Q(G)) = 
                           \begin{cases}
			   0, & \text{ if } \chi(G) \neq 0, \\ 
			   \mathbb{Z}, & \text{ if } \chi(G) = 0, 
			   \end{cases}	
\end{array} 		   
\end{equation}
and $[1_{C^*_Q(G)}]_0$ generates ${\bf K}_0(C^*_Q(G))$ in all cases. 
\par
Moreover ${\bf K}_0(C^*(G)) = \mathbb{Z}$, ${\bf K}_1(C^*(G)) = 0$ and $[1_{C^*(G)}]_0$ generates ${\bf
K}_0(C^*(G))$ in all cases.
\end{prop}

\begin{proof}
We will use induction on the number of vertices of $G$. 
If $G$ has two vertices (and no edges) then $C^*_Q(G) = \mathcal{O}_2$
and $C^*(G) = \mathcal{E}_2$ and in this case certainly the statement is true. 
Suppose that the statement is true
for all graphs $G$ with at most $n \geq 2$ vertices and with $G^{\mathrm{opp}}$ connected. 
The graph $\Gamma$ considered
above was a randomly chosen graph with $n+1$ vertices and with the property that $\Gamma^{\mathrm{opp}}$ is 
connected. 
If we show that the statement holds for $\Gamma$ than this will prove the statement by induction.
\par
We note that from Lemma \ref{lemma:T} and the assumption follows that ${\bf K}_0(T_m) = \mathbb{Z}
[P_m]_0$ and ${\bf K}_1(T_m) = 0$ for all $m \in \mathbb{N}_0$. 
Also since $\mathcal{I}_m \cong \mathcal{K}$ we have ${\bf K}_0(\mathcal{I}_m)= \mathbb{Z} [P_m Q]_0$ and 
${\bf K}_1(\mathcal{I}_m) = 0$ for all $m \in \mathbb{N}_0$. 
Finally we remind that $\pi_m$ is an isomorphism for all $m \in \mathbb{N}_0$. 
\par
From the ${\bf K}$-theory six term eact sequences for the two exact rows of (\ref{equ:3}) we have the following
commutative diagram:

\begin{equation} \label{equ:crit}
\begin{array}{ccccccccc}
{\bf K}_0(\mathcal{I}_{m-1}) & & \overset{{i'_m}_*}{\rightarrow} &  & {\bf K}_0(B_{m-1}) &  & 
   \overset{{p'_m}_*}{\rightarrow} & & {\bf K}_0(\frac{B_{m-1}}{\mathcal{I}_{m-1}}) \\ 
                   &&&&&&&& \\ 
 & \searrow {I'_m}_* & & & \downarrow {I_m}_* & & & \overset{\cong}{\swarrow} {\pi_m}_* &  \\ 
                   &&&&&&&& \\ 
 & & {\bf K}_0(T_m) & \overset{{i_m}_*}{\rightarrow} & {\bf K}_0(B_m) & \overset{{p_m}_*}{\rightarrow} & 
   {\bf K}_0(\frac{B_m}{T_m}) & & \\ 
                   &&&&&&&& \\ 
\uparrow \gamma_m^{\ind} & & \uparrow \delta_m^{\ind} & & & & \downarrow & & 
\downarrow \\ 
                   &&&&&&&& \\ 
 & & {\bf K}_1(\frac{B_m}{T_m}) & \overset{{p_m}_*}{\leftarrow} & {\bf K}_1(B_m) & \leftarrow
   & 0 & & \\ 
                   &&&&&&&& \\ 
 & \overset{\cong}{\nearrow} {\pi_m}_* & & & \uparrow {I_m}_* & & & \nwarrow &  \\ 
                   &&&&&&&& \\ 
{\bf K}_1(\frac{B_{m-1}}{\mathcal{I}_{m-1}}) & & \overset{{p'_m}_*}{\leftarrow} & & {\bf K}_1(B_{m-1}) & &
   \leftarrow & & 0,
\end{array}
\end{equation}
where $\gamma_m^{\ind}$ and $\delta_m^{\ind}$ are the index maps for the corresponding six term exact sequences.
\par
Since $\mathcal{I}_{m-1}$ is generated by $P_m Q$ from Lemma \ref{lemma:3.1} follows that the map 
${i_m}_* : {\bf K}_0(\mathcal{I}_{m-1}) \to {\bf K}_0(B_{m-1})$ is induced by 
$[P_m Q]_{{\bf K}_0(\mathcal{I}_{m-1})} \mapsto \chi(\Gamma') [P_{m-1}]_{{\bf K}_0(B_{m-1})}$. 
Also the map 
${I'_m}_* : {\bf K}_0(\mathcal{I}_{m-1}) \to {\bf K}_0(T_m)$ is induced by 
$[P_m Q]_{{\bf K}_0(\mathcal{I}_{m-1})} \mapsto \chi(\Gamma_k) [P_m]_{{\bf K}_0(T_m)}$.
\par
When we "apply" $\beta$ to equations (\ref{equ:1}) and (\ref{equ:2}) we obtain the following commutative 
diagrams with exact rows:
\begin{equation} \label{equ:b1}
 \begin{CD}
    0 @>>> \mathcal{I}_{m-1} @>{i_m'}>> B_{m-1} @>{p_m'}>> B_{m-1}/\mathcal{I}_{m-1} @>>> 0 \\
           &&    @V{\beta}VV          @V{\beta}VV             @VV{\bar{\beta}}V      &&  \\
    0 @>>> \mathcal{I}_m @>{i_{m+1}'}>> B_m @>{p_{m+1}'}>>  B_m /\mathcal{I}_m  @>>> 0
 \end{CD}
\end{equation}

and

\begin{equation} \label{equ:b2}
 \begin{CD}
    0 @>>> T_m @>{i_m}>> B_m @>{p_m}>> B_m/T_m @>>> 0 \\
           &&    @V{\beta}VV          @V{\beta}VV             @VV{\bar{\bar{\beta}}}V      &&  \\
    0 @>>> T_{m+1} @>{i_{m+1}}>> B_{m+1} @>{p_{m+1}}>> B_{m+1} /T_{m+1} @>>> 0,
 \end{CD}
\end{equation}
where $\bar{\beta}$ and $\bar{\bar{\beta}}$ are induced by $\beta$ on the above quotients. \\ 
\par
We can now start examining the five different cases depending on $\chi(\Gamma')$ and $\chi(\Gamma_k)$:
\par
(case {\bf I}): $\chi(\Gamma') = 0$ and $\chi(\Gamma_k) = 0$. 
\par
By assumption ${i'_m}_* = 0 = {I'_m}_*$. From (\ref{equ:crit}) is easy to see that $\delta_m^{\ind} = 0$.
Therefore (\ref{equ:crit}) splits into two:

\begin{equation} \label{caseI:1}
 \begin{CD}
   & & \dots @>{{i_m'}_*=0}>> {\bf K}_0(B_{m-1}) @>{{p_m'}_*}>{\cong}> {\bf K}_0(B_{m-1}/\mathcal{I}_{m-1}) 
   @>>> 0 \\
           &&&&    @V{{I_m}_*}VV          @V{{\pi_m}_*}V{\cong}V             &&  \\
    @>{\delta_m^{\ind} = 0}>>  {\bf K}_0(T_m)  @>{{i_m}_*}>> {\bf K}_0(B_m) @>{{p_m}_*}>> {\bf K}_0(B_m /T_m) 
    @>>> 0, \\ 
 \end{CD}
\end{equation}

\begin{equation} \label{caseI:2}
 \begin{CD}
   0 @>>> {\bf K}_1(B_{m-1}) @>{{p_m'}_*}>> {\bf K}_1(B_{m-1}/\mathcal{I}_{m-1}) @>{\gamma_m^{\ind}}>> 
   {\bf K}_0(\mathcal{I}_{m-1}) @>{{i'_m}_* = 0}>> \dots \\
        &&    @V{{I_m}_*}VV          @V{{\pi_m}_*}V{\cong}V             &&&  \\
   0 @>>> {\bf K}_1(B_m) @>{{p_m}_*}>{\cong}> {\bf K}_1(B_m /T_m) @>{\delta_m^{\ind} = 0}>> \dots & & &. 
 \end{CD}
\end{equation}

Suppose by induction that ${\bf K}_0(B_{m-1}) = \mathbb{Z} [P_0]_0 \oplus \dots \oplus \mathbb{Z} [P_{m-1}]_0$.
Notice that for $m=1$ we have ${\bf K}_0(B_{0}) = \mathbb{Z} [P_0]_0$. 
Then from (\ref{caseI:1}) follows that 
${\bf K}_0(B_m) = {I_m}_*({\bf K}_0(B_{m-1})) \oplus {i_m}_*({\bf K}_0(T_m))$ since all extensions of free 
abelian groups are trivial. 
Noting that ${\bf K}_0(T_m) = \mathbb{Z} [P_m]_0$ concludes the induction. 
Therefore ${\bf K}_0(B_m) = \mathbb{Z} [P_0]_0 \oplus \dots \oplus \mathbb{Z} [P_m]_0$ for each 
$m \in \mathbb{N}$. 
Notice that we can write ${\bf K}_0(B_m) = \mathbb{Z} [P_0]_0 \oplus \dots \oplus \mathbb{Z} \beta^m_*([P_0]_0)$
\par
Suppose by induction that 
\begin{multline*}
{\bf K}_1(B_{m-1}) = \mathbb{Z} ({p_1}_*)^{-1} \circ {\pi_1}_* \circ (\gamma_1^{\ind})^{-1}([P_1 Q]_0) 
\oplus \dots \\ 
\dots \oplus \mathbb{Z} ({p_{m-1}}_*)^{-1} \circ {\pi_{m-1}}_* \circ
(\gamma_{m-1}^{\ind})^{-1}([P_{m-1} Q]_0). 
\end{multline*}
This is trivially true for $m=1$. 
From (\ref{caseI:2}) we see that ${\bf K}_1(B_m) = ({p_m}_*)^{-1} \circ 
{\pi_m}_*({\bf K}_1(B_{m-1}/\mathcal{I}_{m-1}))$. 
Since all groups are free abelian all the extensions are trivial and therefore 
${\bf K}_1(B_m) = {I_m}_*({\bf K}(B_{m-1})) \oplus \mathbb{Z} ({p_m}_*)^{-1} \circ {\pi_m}_* 
\circ (\gamma_m^{\ind})^{-1}([P_m Q]_0)$. 
This concludes the induction. 
From the functoriality of the index map
and from equations (\ref{equ:b1}) and (\ref{equ:b2}) follows that $({p_m}_*)^{-1} \circ {\pi_m}_* 
\circ (\gamma_m^{\ind})^{-1}([P_m Q]_0) = ({p_m}_*)^{-1} \circ {\pi_m}_* 
\circ (\gamma_m^{\ind})^{-1} \circ \beta_*([P_{m-1}Q]_0) = \beta_* \circ ({p_{m-1}}_*)^{-1} \circ {\pi_{m-1}}_* 
\circ (\gamma_{m-1}^{\ind})^{-1}([P_{m-1}Q]_0)$. 
Therefore we can write 
\begin{multline*}
{\bf K}_1(B_m) = \mathbb{Z} ({p_1}_*)^{-1} \circ {\pi_1}_* \circ (\gamma_1^{\ind})^{-1}([P_1 Q]_0) 
\oplus \dots \\ 
\dots \oplus \mathbb{Z} \beta^{m-1}_*\circ ({p_1}_*)^{-1} \circ {\pi_1}_* \circ (\gamma_1^{\ind})^{-1}([P_1 Q]_0).
\end{multline*}
If $u \in B_1$ is a unitary with $[u]_1 = ({p_1}_*)^{-1} \circ {\pi_1}_* \circ (\gamma_1^{\ind})^{-1}([P_1 Q]_0)$
then we can write 
$${\bf K}_1(B_m) = \mathbb{Z} [u]_1 \oplus \dots \oplus \mathbb{Z} \beta^{m-1}_*([u]_1).$$

From this we get ${\bf K}_0(B) = \underset{i=0}{\overset{\infty}{\oplus}} \mathbb{Z} \beta^i_*([I]_0)$ and 
${\bf K}_1(B) = \underset{i=1}{\overset{\infty}{\oplus}} \mathbb{Z} \beta^{i-1}_*([u]_1)$. \\ 
Let $\tilde{u} = \alpha^0 \circ j_0(u) \in \tilde{B}$. 
Then it is easy to see that ${\bf K}_0(\tilde{B}) = 
\underset{i=-\infty}{\overset{\infty}{\oplus}} \mathbb{Z} \Phi^i_*([\tilde{P}_0]_0))$ and 
${\bf K}_1(\tilde{B}) = \underset{i=-\infty}{\overset{\infty}{\oplus}} \mathbb{Z} \Phi^{i-1}_*([\tilde{u}]_1)$. 
\par
The Pimsner-Voiculescu gives 
\begin{equation*}
 \begin{CD}
    \underset{i=-\infty}{\overset{\infty}{\oplus}} \mathbb{Z} \Phi^i_*([\tilde{P}_0]_0)) 
    @>{\id_* - \Phi_*}>> \underset{i=-\infty}{\overset{\infty}{\oplus}} \mathbb{Z} \Phi^i_*([\tilde{P}_0]_0)) @>>>
    {\bf K_0}(\tilde{A}) \\
                    @AAA            &&                             @VVV        \\
    {\bf K_1}(\tilde{A}) @<<< \underset{i=-\infty}{\overset{\infty}{\oplus}} \mathbb{Z} 
    \Phi^{i-1}_*([\tilde{u}]_1) 
    @<{\id_* - \Phi_*}<< \underset{i=-\infty}{\overset{\infty}{\oplus}} \mathbb{Z} \Phi^{i-1}_*([\tilde{u}]_1). \\
 \end{CD}
\end{equation*}

From this we can conclude that ${\bf K}_0(\tilde{A}) = \mathbb{Z} [\tilde{P}_0]_0$, ${\bf K}_1(\tilde{A}) =
\mathbb{Z}$. 
From Proposition \ref{prop:4} follows that ${\bf K}_0(\tilde{A}) \cong 
{\bf K}_0(C^*_Q(\Gamma)) = \mathbb{Z}$ and ${\bf K}_1(\tilde{A}) \cong 
{\bf K}_1(C^*_Q(\Gamma)) = \mathbb{Z}$ and that $[1_{C^*_Q(\Gamma)}]_0$ generates ${\bf K}_0(C^*_Q(\Gamma))$. 
\par
From Remark \ref{remark} follows that in the extension (\ref{equ:remark}) the map ${I_{\Gamma}}_*$ on ${\bf K}_0$
is zero. 
This shows that ${\bf K}_0(C^*(\Gamma)) = \mathbb{Z} [1_{C^*(\Gamma)}]_0$ and ${\bf K}_1(C^*(\Gamma)) =
0$. 
\par
This concludes the proof of (case {\bf I}).
\par
(case {\bf II}): $\chi(\Gamma') \neq 0$ and $\chi(\Gamma_k) = 0$. 
\par
By assumption ${\bf K}_0(B_0) = \mathbb{Z} [P_0]_0$, ${\bf K}_0(B_0/\mathcal{I}_0) = \mathbb{Z}_{|\chi(\Gamma')|} 
{p_1'}_*([P_0]_0)$, ${\bf K}_1(B_0) = 0$ and ${\bf K}_1(B_0/\mathcal{I}_0) = 0$. 
\par 
Suppose by induction that 
$${\bf K}_0(B_{m-1}) = \mathbb{Z} [P_{m-1}]_0 \oplus \mathbb{Z}_{|\chi(\Gamma')|}
[P_{m-2}]_0 \oplus \dots \oplus \mathbb{Z}_{|\chi(\Gamma')|} [P_0]_0,$$ 
$${\bf K}_0(B_{m-1}/\mathcal{I}_{m-1}) = 
\mathbb{Z}_{|\chi(\Gamma')|} {p_m'}_*([P_{m-1}]_0) \oplus \mathbb{Z}_{|\chi(\Gamma')|} {p_m'}_*([P_{m-2}]_0) 
\oplus \dots \oplus \mathbb{Z}_{|\chi(\Gamma')|} {p_m'}_*([P_0]_0),$$ 
${\bf K}_1(B_{m-1}) = 0$ and ${\bf K}_1(B_{m-1}/\mathcal{I}_{m-1}) = 0$. \\ 
Then from diagram (\ref{equ:crit}) immediately follows that ${\bf K}_1(B_m) = 0$ and ${\bf K}_1(B_m/T_m) = 0$. 
\par
Then (\ref{equ:crit}) reduces to the following commutative diagram with exact rows:
\begin{equation} \label{caseII}
 \begin{CD}
    0@>>> {\bf K}_0(\mathcal{I}_{m-1}) @>{{i_m'}_*}>> {\bf K}_0(B_{m-1}) @>{{p_m'}_*}>> 
    {\bf K}_0(B_{m-1}/\mathcal{I}_{m-1}) @>>> 0\\
           &&    @V{{I_m'}_* = 0}VV          @V{{I_m}_*}VV             @V{\cong}V{{\pi_m}_*}V      &&  \\
    0 @>>> {\bf K}_0(T_m) @>{{i_m}_*}>> {\bf K}_0(B_m) @>{{p_m}_*}>> {\bf K}_0(B_m/T_m) @>>> 0.
 \end{CD}
\end{equation}

From Lemma \ref{lemma:3.1} we have that $\chi(\Gamma') [P_l]_0 = 0$ in ${\bf K}_0(B_m)$ for $l=0, \dots, m-1$.
Since ${\pi_m}_*$ is an isomorphism then by the inducton hypothesis and (\ref{caseII}) it is easy to see that 
${p_m}_*$ restricted to $\mathcal{G} = \langle [P_0]_0, \dots, [P_{m-1}]_0 \rangle_{{\bf K}_0(B_m)}$ is an 
isomorphism. 
This fact also implies that there are no relations between
$[P_m]_0$ and $\mathcal{G}$ (since the bottom row of (\ref{caseII}) is exact). 
Since ${i_m}_*$ is injective then 
$[P_m]_0$ in of infinite order in ${\bf K}_0(B_m)$. 
Clearly ${\bf K}_0(B_m)$ is generated by $[P_m]_0$ and $\mathcal{G}$.  
Therefore 
$${\bf K}_0(B_m) = \mathbb{Z} [P_m]_0 \oplus \mathbb{Z}_{|\chi(\Gamma')|} [P_{m-1}]_0 \oplus \dots \oplus 
\mathbb{Z}_{|\chi(\Gamma')|} [P_0]_0.$$

From the following six term exact sequence
\begin{equation*}
 \begin{CD}
    {\bf K_0}(\mathcal{I}_m) @>{{i_m'}_*}>> {\bf K_0}(B_m) @>{{p_m'}_*}>> {\bf K_0}(B_m/\mathcal{I}_m) \\
                    @AAA            &&                             @VVV        \\
    {\bf K_1}(B_m/\mathcal{I}_m) @<{{p_m'}_*}<< {\bf K_1}(B_m) @<{{i_m'}_*}<< {\bf K_1}(\mathcal{I}_m) \\
 \end{CD}
\end{equation*}
and the fact that ${i_m'}_*$ is given by $[P_{m+1}Q]_0 \mapsto \chi(\Gamma') [P_m]_0$ we easily get that \\ 
${\bf K}_1(B_m/\mathcal{I}_m) = 0$ and that 
$${\bf K}_0(B_m/\mathcal{I}_m) = \mathbb{Z}_{|\chi(\Gamma')|} {p_{m+1}'}_*([P_m]_0) \oplus 
\mathbb{Z}_{|\chi(\Gamma')|} {p_{m+1}'}_*([P_{m-1}]_0) \oplus \dots \oplus 
\mathbb{Z}_{|\chi(\Gamma')|} {p_{m+1}'}_*([P_0]_0).$$
This completes the induction.
\par
We have ${\bf K}_0(B) = \underset{i=0}{\overset{\infty}{\oplus}} \mathbb{Z}_{|\chi(\Gamma')|} [P_i]_0$ 
and ${\bf K}_1(B) = 0$. 
We can write 
${\bf K}_0(B) = \underset{i=0}{\overset{\infty}{\oplus}} \mathbb{Z}_{|\chi(\Gamma')|} \beta^i_*([P_0]_0)$. 
Therefore ${\bf K}_0(\tilde{B}) = \underset{i=-\infty}{\overset{\infty}{\oplus}} \mathbb{Z}_{|\chi(\Gamma')|}
\Phi^i_*([\tilde{P}_0]_0)$ and ${\bf K}_1(\tilde{B}) = 0$. 
\par
The Pimsner-Voiculescu exact sequence gives
\begin{equation*}
 \begin{CD}
    \underset{i=-\infty}{\overset{\infty}{\oplus}} \mathbb{Z}_{|\chi(\Gamma')|} \Phi^i_*([\tilde{P}_0]_0)) 
    @>{\id_* - \Phi_*}>> \underset{i=-\infty}{\overset{\infty}{\oplus}} \mathbb{Z}_{|\chi(\Gamma')|} 
    \Phi^i_*([\tilde{P}_0]_0)) @>>> {\bf K_0}(\tilde{A}) \\
                    @AAA            &&                             @VVV        \\
    {\bf K_1}(\tilde{A}) @<<< 0 @<<< 0. \\
 \end{CD}
\end{equation*}
We conclude that ${\bf K}_0(\tilde{A}) = \mathbb{Z}_{|\chi(\Gamma')|} [\tilde{P}_0]_0$ and ${\bf K}_1(\tilde{A})
= 0$. 
From Proposition \ref{prop:4} we get ${\bf K}_0(C^*_Q(\Gamma)) = \mathbb{Z}_{|\chi(\Gamma)|}
[1_{C^*_Q(\Gamma)}]_0$ and ${\bf K}_1(C^*_Q(\Gamma)) = 0$ (notice that $\chi(\Gamma) = \chi(\Gamma') -
\chi(\Gamma_k) = \chi(\Gamma') - 0$). 
\par
From Remark \ref{remark} follows that ${I_{\Gamma}}_*$ is "multiplication by $\chi(\Gamma)$", so 
${\bf K}_0(C^*(\Gamma)) = \mathbb{Z} [1_{C^*(\Gamma)}]_0$ and ${\bf K}_1(C^*(\Gamma)) = 0$. 
\par
This concludes the proof of (Case {\bf II}).
\par
(case {\bf III}): $\chi(\Gamma') = 0$ and $\chi(\Gamma_k) \neq 0$. 
\par
By assumption we have that ${\bf K}_0(B_0) = \mathbb{Z} [P_0]_0$, ${\bf K}_0(B_0/\mathcal{I}_0) = \mathbb{Z}
{p_1'}_*([P_0]_0)$, ${\bf K}_1(B_0) = 0$ and ${\bf K}_1(B_0/\mathcal{I}_0) = \mathbb{Z}$. 
\par
Suppose by induction that 
$${\bf K}_0(B_{m-1}) = \mathbb{Z} [P_0]_0 \oplus \mathbb{Z}_{|\chi(\Gamma_k)|} [P_1]_0 \oplus \dots \oplus 
\mathbb{Z}_{|\chi(\Gamma_k)|} [P_{m-1}]_0$$ 
and that ${\bf K}_1(B_{m-1}) = 0$. 
Then from diagram (\ref{equ:crit}) we see that the maps $\gamma_m^{\ind}$ and ${p_m'}_*: {\bf K}_0(B_{m-1}) \to
{\bf K}_0(B_{m-1}/\mathcal{I}_{m-1})$ are isomorphisms. 
Since also $\pi_m$ is an isomorphism this implies that
${I_m}_*: {\bf K}_0(B_{m-1}) \to {\bf K}_0(B_m)$ is an isomorphism into and that ${p_m}_*$ restricted to 
$\mathcal{G} = \langle [P_0]_0, \dots, [P_{m-1}]_0 \rangle_{{\bf K}_0(B_m)}$ is also an
isomorphism. 
From the fact that ${p_m}_*|_{\mathcal{G}}$ is injective follows that there are no relations 
between $[P_m]_0$ and $\mathcal{G}$. \\ 
The commutativity of 
\begin{equation*}
 \begin{CD}
   {\bf K}_1(B_{m-1}/\mathcal{I}_{m-1}) @>{\gamma_m^{\ind}}>{\cong}> {\bf K}_0(\mathcal{I}_{m-1}) \\
       @V{{\pi_m}_*}V{\cong}V               @V{{I_m}_*}V{[= \times \chi(\Gamma_k)]}V              \\
   {\bf K}_1(B_m /T_m) @>{\delta_m^{\ind}}>> {\bf K}_0(T_m) 
 \end{CD}
\end{equation*}
implies that $\delta_m^{\ind}$ is "multiplication by $\chi(\Gamma_k)$". 
Thus $[P_m]_0$ in ${\bf K}_0(B_m)$ is of
order $\chi(\Gamma_k)$ (as should be by Lemma \ref{lemma:3.1}). 
Therefore
$${\bf K}_0(B_m) = \mathbb{Z} [P_0]_0 \oplus \mathbb{Z}_{|\chi(\Gamma_k)|} [P_1]_0 \oplus \dots \oplus 
\mathbb{Z}_{|\chi(\Gamma_k)|} [P_m]_0.$$
We also showed that $\delta_m^{\ind}$ is injective and therefore ${\bf K}_1(B_m) = 0$. 
\par
Now we easily get ${\bf K}_0(B) = \mathbb{Z} [P_0]_0 \oplus \underset{i=1}{\overset{\infty}{\oplus}} 
\mathbb{Z}_{|\chi(\Gamma_k)|}
\beta^i_*([P_0]_0)$ and ${\bf K}_1(B) = 0$. 
From this follows that ${\bf K}_0(\tilde{B}) = 
\underset{i=-\infty}{\overset{\infty}{\oplus}} \mathbb{Z}_{|\chi(\Gamma_k)|} \Phi^i_*([\tilde{P}_0]_0)$
and ${\bf K}_1(\tilde{B}) = 0$. 
\par
The Pimsner-Voiculescu exact sequence gives
\begin{equation*}
 \begin{CD}
    \underset{i=-\infty}{\overset{\infty}{\oplus}} \mathbb{Z}_{|\chi(\Gamma_k)|} \Phi^i_*([\tilde{P}_0]_0)) 
    @>{\id_* - \Phi_*}>> \underset{i=-\infty}{\overset{\infty}{\oplus}} \mathbb{Z}_{|\chi(\Gamma_k)|} 
    \Phi^i_*([\tilde{P}_0]_0)) @>>> {\bf K_0}(\tilde{A}) \\
                    @AAA            &&                             @VVV        \\
    {\bf K_1}(\tilde{A}) @<<< 0 @<<< 0. \\
 \end{CD}
\end{equation*}
We conclude that ${\bf K}_0(\tilde{A}) = \mathbb{Z}_{|\chi(\Gamma_k)|} [\tilde{P}_0]_0$ and ${\bf K}_1(\tilde{A})
= 0$. 
From Proposition \ref{prop:4} we get ${\bf K}_0(C^*_Q(\Gamma)) = \mathbb{Z}_{|\chi(\Gamma)|}
[1_{C^*_Q(\Gamma)}]_0$ and ${\bf K}_1(C^*_Q(\Gamma)) = 0$ (notice that $\chi(\Gamma) = \chi(\Gamma') -
\chi(\Gamma_k) = 0 - \chi(\Gamma_k)$).
\par
From Remark \ref{remark} follows that ${I_{\Gamma}}_*$ is "multiplication by $\chi(\Gamma)$", so 
${\bf K}_0(C^*(\Gamma)) = \mathbb{Z} [1_{C^*(\Gamma)}]_0$ and ${\bf K}_1(C^*(\Gamma)) = 0$. 
\par
This concludes the proof of (Case {\bf III}).
\par
(cases {\bf IV} and {\bf V}): $\chi(\Gamma') \neq 0$, $\chi(\Gamma_k) \neq 0$.
\par 
Let's denote $x = \chi(\Gamma')$, $y = \chi(\Gamma_k)$ and let $\GCD(x,y) = d > 0$ be the greatest common divisor 
of $x$ and $y$. 
Then by the B$\acute{e}$zout's identity there exist $a,b \in \mathbb{Z}$ such that $ax + by = d$.
Denote also $x' = x/d$ and $y' = y/d$. 
Then $ax' + by' = 1$. 
\par
By assumption we have that ${\bf K}_0(B_0) = \mathbb{Z} [P_0]_0$, ${\bf K}_1(B_0/\mathcal{I}_0) = 0$, 
${\bf K}_1(B_0) = 0$ and ${\bf K}_0(B_0/\mathcal{I}_0) = \mathbb{Z}_{|x|} {p_1'}_*([P_0]_0)$. 
\par
Then $0 = {\bf K}_1(B_0/\mathcal{I}_0) \cong {\bf K}_1(B_1/T_1)$ which implies ${\bf K}_1(B_1) = 0$.
Diagram (\ref{equ:crit}) for $m=1$ reduces to 
\begin{equation} \label{caseV:1}
 \begin{CD}
    0@>>> {\bf K}_0(\mathcal{I}_0) @>{{i_1'}_*}>> {\bf K}_0(B_0) @>{{p_1'}_*}>> 
    {\bf K}_0(B_0/\mathcal{I}_0) @>>> 0\\
           &&    @V{{I_1'}_*}VV          @V{{I_1}_*}VV             @V{\cong}V{{\pi_1}_*}V      &&  \\
    0 @>>> {\bf K}_0(T_1) @>{{i_1}_*}>> {\bf K}_0(B_1) @>{{p_1}_*}>> {\bf K}_0(B_1/T_1) @>>> 0.
 \end{CD}
\end{equation}

Clearly in ${\bf K}_0(B_1)$ we have $x[P_0]_0 - y[P_1]_0 = 0$. 
Consider $g = b[P_0]_0 + a[P_1]_0,\ 
g' = x'[P_0]_0 - y'[P_1]_0 \in {\bf K}_0(B_1)$. 
Since $a g' + y' g = (a x'+y' b)[P_0]_0 = [P_0]_0$ and $x' g-b g' =
(x' a+b y')[P_1]_0 = [P_1]_0$ it follows that $g$ and $g'$ generate ${\bf K}_0(B_1)$. 
Since ${i_1}_*$ is injective 
on ${\bf K}_0$ it follows that ${\bf K}_0(B_1)$ is an infinite group. 
Clearly 
$dg' = dx'[P_0]_0 - dy'[P_1]_0 = 0$. 
Therefore $g$ is of infinite order in ${\bf K}_0(B_1)$ and moreover $g$ and
$g'$ are not related (or otherwise $g$ would be of finite order). 
If we suppose that
$0 < d'\ |\ d$ and $d' g' = 0$ then it will follow that $d'x'[P_0]_0 = d' y' [P_1]_0 \in \ker({p_1}_*)$. 
But the order of ${p_1}_*([P_0]_0)$ in ${\bf K}_0(B_1/T_1)$ is $|x|$, so $d' x' \geq x$ or $d' \geq d$. 
Therefore $d' = d$ and ${\bf K}_0(B_1) =
\mathbb{Z} (b[P_0]_0 + a[P_1]_0) \oplus \mathbb{Z}_d (x'[P_0]_0 - y'[P_1]_0)$. 
\par
Therefore we showed that ${\bf K}_0(B_1) = \{ \mathbb{Z} [P_0]_0 \oplus \mathbb{Z} 
[P_1]_0 | x[P_0]_0 - y[P_1]_0 = 0 \}$. 
Suppose by induction that for $m \geq 2$, ${\bf K}_1(B_{m-1}) = 0$ and that 
$${\bf K}_0(B_{m-1})  = \{ \mathbb{Z} [P_0]_0 \oplus \dots \oplus \mathbb{Z} [P_{m-1}]_0 | 
x[P_0]_0 - y[P_1]_0 = 0, 
\dots, x[P_{m-2}]_0 - y[P_{m-1}]_0 = 0 \}.$$

From the induction hypothesis follows that $[P_{m-1}]_0$ is of infinite order in ${\bf K}_0(B_{m-1})$ and
therefore that ${i_m'}_*: {\bf K}_0(\mathcal{I}_{m-1}) \to {\bf K}_0(B_{m-1})$ is injective and therfore $0 = 
{\bf K}_1(B_{m-1}/\mathcal{I}_{m-1}) \cong {\bf K}_1(B_m/T_m)$. 
This shows that ${\bf K}_1(B_m) = 0$ and that 
(\ref{equ:crit}) reduces to 
\begin{equation} \label{caseV:m}
 \begin{CD}
    0@>>> {\bf K}_0(\mathcal{I}_{m-1}) @>{{i_m'}_*}>> {\bf K}_0(B_{m-1}) @>{{p_m'}_*}>> 
    {\bf K}_0(B_{m-1}/\mathcal{I}_{m-1}) @>>> 0\\
           &&    @V{{I_m'}_*}VV          @V{{I_m}_*}VV             @V{\cong}V{{\pi_m}_*}V      &&  \\
    0 @>>> {\bf K}_0(T_m) @>{{i_m}_*}>> {\bf K}_0(B_m) @>{{p_m}_*}>> {\bf K}_0(B_m/T_m) @>>> 0.
 \end{CD}
\end{equation} 

It is easy to see that 
\begin{multline*}
{\bf K}_0(B_{m-1}/\mathcal{I}_{m-1}) = \{\mathbb{Z} {p_m'}_*([P_{m-1}]_0) \oplus \dots \oplus \mathbb{Z}
{p_m'}_*([P_0]_0) | x{p_m'}_*([P_{m-1}]_0) = 0, \\ 
x{p_m'}_*([P_{m-2}]_0) - y{p_m'}_*([P_{m-1}]_0) = 0, \dots, x{p_m'}_*([P_0]_0) - y{p_m'}_*([P_1]_0) = 0 \}.
\end{multline*}
Since ${I_m'}_*$ is "multiplication by $\chi(\Gamma_k)$" and therefore injective then by the Five Lemma follows 
that ${I_m}_*$ is also injective. 
Therefore if we denote $\mathcal{G} = {I_m}_*({\bf K}_0(B_{m-1}))$ then 
${\bf K}_0(B_m) = \langle [P_m]_0, \mathcal{G} \rangle$. 
One obvious relation in ${\bf K}_0(B_m)$ beside the 
relations that come from ${\bf K}_0(B_{m-1})$ is $x[P_{m-1}]_0 - y[P_m]_0 = 0$ and this relation follows from 
Lemma \ref{lemma:3.1}. \\ 
Therefore ${\bf K}_0(B_m)$ is a quotient of the group  
$$F= \{\mathbb{Z} \rho_0 \oplus \dots \oplus \mathbb{Z} \rho_m | x\rho_{m-1} - y\rho_m = 0, \dots, 
x\rho_0 - y\rho_1 = 0 \},$$
where the quotient map $f : F \to {\bf K}_0(B_m)$ is defined on the generators as $\rho_l \mapsto [P_l]_0$, $l =
0, \dots, m$.  
Then if $F' = \mathbb{Z} \rho_m$ the quotient $F_q = F/F'$ is isomorphis to 
\begin{multline*}
F_q = \{ \mathbb{Z} \rho_0 \oplus \dots \oplus \mathbb{Z} \rho_m | x\rho_{m-1} - y\rho_m = 0, \dots, 
x\rho_0 - y\rho_1 = 0, \rho_m = 0 \} = \\ 
= \{ \mathbb{Z} \rho_0 \oplus \dots \oplus \mathbb{Z} \rho_{m-1} | x\rho_{m-1} = 0, x\rho_{m-2} - y\rho_{m-1} =
0, \dots, x\rho_0 - y\rho_1 = 0 \}.
\end{multline*}

Obviously we have the commutative diagram of abelian groups with exact rows
\begin{equation*} 
 \begin{CD}
    0 @>>> F' @>>> F @>>> F_q @>>> 0 \\ 
      &&    @V{f'}VV          @V{f}VV             @V{f_q}VV      &&  \\
    0 @>>> {\bf K}_0(T_m) @>{{i_m}_*}>> {\bf K}_0(B_m) @>{{p_m}_*}>> {\bf K}_0(B_m/T_m) @>>> 0,
 \end{CD}
\end{equation*}
where $f_q$ is the homomorphism induced by $f$ and $f'$ is the restriction of $f$ to $F'$. 
Then obviously $f'$
and $f_q$ are isomorphisms (since ${\pi_m}_*$ is an isomorphism). 
Therefore by the Five Lemma follows that $f$ is also an isomorphism. 
\par
This shows that 
$${\bf K}_0(B_m) = \{ \mathbb{Z} [P_0]_0 \oplus \dots \oplus \mathbb{Z} [P_m]_0 | 
x[P_0]_0 - y[P_1]_0 = 0, \dots, x[P_{m-1}]_0 - y[P_m]_0 = 0 \}.$$
We also showed above that ${\bf K}_1(B_m) = 0$ and this concludes the induction. 
\par
Now it is easy to see that ${\bf K}_1(B) = 0$ and that 
$${\bf K}_0(B) = \{ \underset{i=0}{\overset{\infty}{\oplus}} \mathbb{Z} \beta^i_*([P_0]_0) | \chi(\Gamma')
\beta^i_*([P_0]_0) - \chi(\Gamma_k) \beta^{i+1}_*([P_0]_0) = 0,\ i \in \mathbb{N}_0 \}.$$
Then ${\bf K}_1(\tilde{B}) = 0$ and 
$${\bf K}_0(\tilde{B}) = \{ \underset{i=-\infty}{\overset{\infty}{\oplus}} \mathbb{Z} \Phi^i_*([\tilde{P}_0]_0) | 
\chi(\Gamma') \Phi^i_*([\tilde{P}_0]_0) - \chi(\Gamma_k) \Phi^{i+1}_*([\tilde{P}_0]_0) = 0,\ i \in 
\mathbb{Z} \}.$$ 

The Pimsner-Voiculescu exact sequence gives
\begin{equation} \label{cases45}
 \begin{CD}
    {\bf K}_0(\tilde{B}) 
    @>{\id_* - \Phi_*}>> 
    {\bf K}_0(\tilde{B}) @>>> {\bf K_0}(\tilde{A}) \\
                    @AAA            &&                             @VVV        \\
    {\bf K_1}(\tilde{A}) @<<< 0 @<<< 0. \\
 \end{CD}
\end{equation}

(Case {\bf IV}): $\chi(\Gamma') \neq 0$, $\chi(\Gamma_k) \neq 0$ and $\chi(\Gamma') = \chi(\Gamma_k)$.
\par
In this case 
\begin{multline*}
{\bf K}_0(\tilde{A}) = \{ \underset{i=-\infty}{\overset{\infty}{\oplus}} \mathbb{Z} \Phi^i_*([\tilde{P}_0]_0) | 
\chi(\Gamma') \Phi^i_*([\tilde{P}_0]_0) - \chi(\Gamma') \Phi^{i+1}_*([\tilde{P}_0]_0) = 0,\\  
\Phi^i_*([\tilde{P}_0]_0) - \Phi^{i+1}_*([\tilde{P}_0]_0) = 0,\ i \in \mathbb{Z}  \} = \\ 
\{ \underset{i=-\infty}{\overset{\infty}{\oplus}} \mathbb{Z} \Phi^i_*([\tilde{P}_0]_0) |   
\Phi^i_*([\tilde{P}_0]_0) - \Phi^{i+1}_*([\tilde{P}_0]_0) = 0,\ i \in \mathbb{Z}  \} = 
\mathbb{Z} [\tilde{P}_0]_0. 
\end{multline*}

To examine $\ker(\id_* - \Phi_*)$ take $\omega = \underset{i=-j}{\overset{j}{\Sigma}} t_i 
\Phi^i_*([\tilde{P}_0]_0) \in \ker(\id_* - \Phi_*)$, where $t_i \in \mathbb{Z}$. 
Then
$$0 = (\id_* - \Phi_*)(\omega) = \underset{i=-j}{\overset{j}{\Sigma}} t_i 
(\id_* - \Phi_*)(\Phi^i_*([\tilde{P}_0]_0)) =
\underset{i=-j}{\overset{j}{\Sigma}} t_i (\Phi^i_*([\tilde{P}_0]_0) - \Phi^{i+1}_*([\tilde{P}_0]_0)).$$
Therefore $t_i = s_i|\chi(\Gamma')|$ for some integers $s_i$, $i=-j, \dots, j$. 
From this easily follows that
$\omega = \underset{i=-j}{\overset{j}{\Sigma}} s_i|\chi(\Gamma')| [\tilde{P}_0]_0.$ 
Thus $\ker(\id_* - \Phi_*) =
|\chi(\Gamma')| \mathbb{Z} [\tilde{P}_0]_0$. 
This shows that ${\bf K}_1(\tilde{A}) = \mathbb{Z}$.
\par
From Proposition \ref{prop:4} follows that ${\bf K}_0(\tilde{A}) \cong 
{\bf K}_0(C^*_Q(\Gamma)) = \mathbb{Z}$, ${\bf K}_1(\tilde{A}) \cong 
{\bf K}_1(C^*_Q(\Gamma)) = \mathbb{Z}$ and that $[1_{C^*_Q(\Gamma)}]_0$ generates ${\bf K}_0(C^*_Q(\Gamma))$. 
 \par
From Remark \ref{remark} follows that in the extension (\ref{equ:remark}) the map ${I_{\Gamma}}_*$ on ${\bf K}_0$
is zero (since $\chi(\Gamma) = \chi(\Gamma') - \chi(\Gamma_k) = 0$). 
Therefore ${\bf K}_0(C^*(\Gamma)) = \mathbb{Z} [1_{C^*(\Gamma)}]_0$ and ${\bf K}_1(C^*(\Gamma)) = 0$. 
\par
This concludes the proof of (case {\bf IV}).
\par
(case {\bf V}): $\chi(\Gamma') \neq 0$, $\chi(\Gamma_k) \neq 0$ and $\chi(\Gamma') \neq \chi(\Gamma_k)$. 
\par
In this case  
\begin{multline*}
{\bf K}_0(\tilde{A}) = \{ \underset{i=-\infty}{\overset{\infty}{\oplus}} \mathbb{Z} \Phi^i_*([\tilde{P}_0]_0) | 
\chi(\Gamma') \Phi^i_*([\tilde{P}_0]_0) - \chi(\Gamma_k) \Phi^{i+1}_*([\tilde{P}_0]_0) = 0,\\  
\Phi^i_*([\tilde{P}_0]_0) - \Phi^{i+1}_*([\tilde{P}_0]_0) = 0,\ i \in \mathbb{Z}  \} = \\ 
\{ \mathbb{Z} [\tilde{P}_0]_0 | \chi(\Gamma') [\tilde{P}_0]_0 - \chi(\Gamma_k) [\tilde{P}_0]_0 = 0 \} =
\mathbb{Z}_{|\chi(\Gamma') - \chi(\Gamma_k)|} [\tilde{P}_0]_0 = \mathbb{Z}_{|\chi(\Gamma)|} [\tilde{P}_0]_0. 
\end{multline*}

We only need to show that ${\bf K}_1(\tilde{A}) = 0$ or that $\id_* - \Phi_*$ is injective. \\ 
Take $\omega = \underset{i=-j}{\overset{j}{\Sigma}} t_j \Phi^i_*([\tilde{P}_0]_0)$, $t_i \in \mathbb{Z}$ and 
suppose that $(\id_* - \Phi_*)(\omega) = 0$. 
Then
\begin{multline*}
0 = (\id_* - \Phi_*)(\omega) = \underset{i=-j}{\overset{j}{\Sigma}} t_j (\Phi^i_*([\tilde{P}_0]_0) - 
\Phi^{i+1}_*([\tilde{P}_0]_0)) = \\ 
 = t_{-j} \Phi^{-j}_*([\tilde{P}_0]_0) + \underset{i=-j+1}{\overset{j}{\Sigma}}
(t_i - t_{i-1}) \Phi^i_*([\tilde{P}_0]_0) - t_j \Phi^{j+1}_*([\tilde{P}_0]_0). 
\end{multline*}
If $\chi(\Gamma')$ doesn't divide $t_{-j}$ then the equality $-t_{-j} \Phi^{-j}_*([\tilde{P}_0]_0) = (t_i -
t_{i-1}) \Phi^i_*([\tilde{P}_0]_0) - t_j \Phi^{j+1}_*([\tilde{P}_0]_0)$ is impossible. 
If $\chi(\Gamma')$ divides
$t_{-j}$ then $\omega$ can be expressed in terms of $\Phi^{-j+1}_*([\tilde{P}_0]_0), \dots, 
\Phi^j_*([\tilde{P}_0]_0)$. 
By induction we see that we can write $\omega = t [\tilde{P}_0]_0$ for some $t \in
\mathbb{Z}$. But then clearly $(\id_* - \Phi_*)(\omega) = 0$ is possible if and only if $t=0$. 
This shows that $\id_* - \Phi_*$ is injective and therefore that ${\bf K}_1(\tilde{A}) = 0$. 
\par
From Proposition \ref{prop:4} we get ${\bf K}_0(C^*_Q(\Gamma)) = \mathbb{Z}_{|\chi(\Gamma)|}
[1_{C^*_Q(\Gamma)}]_0$ and ${\bf K}_1(C^*_Q(\Gamma)) = 0$. 
\par
From Remark \ref{remark} follows that ${I_{\Gamma}}_*$ is "multiplication by $\chi(\Gamma)$", so 
${\bf K}_0(C^*(\Gamma)) = \mathbb{Z} [1_{C^*(\Gamma)}]_0$ and ${\bf K}_1(C^*(\Gamma)) = 0$. 
\par
This concludes the proof of (Case {\bf V}). 
\par
The Proposition is proved.
\end{proof}
 
Now we can apply the Kirchberg-Phillips Classification theorem (\cite{Ph00}) to $C^*_Q(G)$ for a finite graph
$G$ such that $G^{\mathrm{opp}}$ is connected and with at least two vertices, using Theorem \ref{thm:1}, 
Proposition 
\ref{prop:4} and Proposition \ref{prop:main}. 
We obtain
\begin{equation}
C^*_Q(G) \cong \mathcal{O}_{1+|\chi(G)|}. 
\end{equation}

For infinite graphs with connected opposite graphs we can argue similarly as in \cite[Corollary 3.11]{C81} to 
prove the following:

\begin{prop} \label{prop:infty}
Let $G$ an infinite graph with countably many vertices and such that $G^{\mathrm{opp}}$ is connected. 
Then $C^*(G)$ ($=C^*_Q(G)$) is nuclear and belongs to the small bootstrap class. Moreover 
${\bf K}_0(C^*(G)) = \mathbb{Z} [1_{C^*(G)}]_0$ and ${\bf K}_1(C^*(G)) = 0$. 
\end{prop}

\begin{proof}
By induction we will find a increasing sequence $G_n$ of subgraphs of $G$ with $n$ 
vertices, $n \geq 2$ which are such that $G_n^{\mathrm{opp}}$ is connected for each $n \geq 2$ and also 
$G_n \overset{n \to \infty}{\longrightarrow} G$. 
Obviously we can find two vertices $v_1$ and $v_2$ that are not connected (since $G^{\mathrm{opp}}$ is conected). 
Then we chose $G_2$ to be the graph with vertices $v_1$ and $v_2$ and no edges. 
Suppose we have defined the subgraph $G_n$ for some $n \geq 2$. 
Let $v_1, \dots, v_n$ be the vertices of $G_n$. 
Since $G^{\mathrm{opp}}$ is connected we can find a vertex $v_{n+1}$
of $G$ different from $v_1, \dots, v_n$ such that $v_{n+1}$ is not connected with all of the vertices $v_1,
\dots, v_n$. 
Then obviously the subgraph $G_{n+1}$ of $G$ on vertices $v_1, \dots, v_{n+1}$ and edges comming
from $G$ is such that $G_{n+1}^{\mathrm{opp}}$ is connected. 
This completes the induction. 
\par 
From Proposition \ref{prop:main} we have ${\bf K}_0(C^*(G_n)) = \mathbb{Z} [1_{C^*(G_n)}]_0$ and ${\bf
K}_1(C^*(G_n)) = 0$. It is easy to see that $C^*(G) = \underset{\longrightarrow}{\lim} C^*(G_n)$.
Therefore from Proposition \ref{prop:4} we get that $C^*_Q(G)$ is nuclear and belongs to the small 
bootstrap category $\mathfrak{N}$. 
Also ${\bf K}_0(C^*(G)) = \underset{n \to \infty}{\lim} {\bf K}_0(C^*(G_n)) = 
\mathbb{Z} [1_{C^*(G)}]_0$ and ${\bf K}_1(C^*(G)) = \underset{n \to \infty}{\lim} {\bf K}_0(C^*(G_n)) = 0$. 
\par
This proves the proposition.
\end{proof}

From Theorem \ref{thm:1} we know that $C^*(G) = C^*_Q(G)$ is purely infinite and simple. 
Again using Kirchberg-Phillips theorem we get that if $G$ is an infinite graph on countably many vertices
such that $G^{\mathrm{opp}}$ is connected then $C^*(G) = C^*_Q(G) \cong \mathcal{O}_{\infty}$. 
If we define for an infinite
countable graph $G$ with $G^{\mathrm{opp}}$ conected $\chi(G) \overset{def}{=} \infty$ then we can write once 
again $C^*_Q(G) \cong \mathcal{O}_{1+|\chi(G)|}$. 

\begin{remark} \label{tensor}
Let $G_1$ and $G_2$ be two disjoint graphs. 
Then by $G_1 * G_2$ we denote their join which is the graph obtained from $G_1$ and $G_2$
by connecting each vertex of $G_1$ with each vertex of $G_2$. 
Then if we start with a graph $G$ on countably many vertices which is such that $G^{\mathrm{opp}}$ doesn't 
have any isolated vertices then we can find a sequence of subraphs
$G_n$, $n \in \mathbb{N}$ (some of $G_n$'s can have zero vertices) such that $G_n^{\mathrm{opp}}$ are all 
connected and such that $G = \underset{n=1}{\overset{\infty}{*}} G_n$. 
For a graph $F$ with zero vertices we write $C^*_Q(F) = \mathbb{C}$. 
\par
Then from Theorem \ref{thm:1} easily follows that $C^*_Q(G) = \underset{n=1}{\overset{\infty}{\otimes}}
C^*_Q(G_n)$. 
\end{remark}

Now we can record our main result:

\begin{thm} 
Let $G$ be a graph with at least two and at most countably many vertices such that $G^{\mathrm{opp}}$ has 
no isolated vertices. 
Write $G = \underset{n=1}{\overset{\infty}{*}} G_n$ as in Remark \ref{tensor} with $G_n$ being a
subgraph of $G$ such that $G_n^{\mathrm{opp}}$ is connected. 
\par
Then
\begin{equation}
C^*_Q(G) = \underset{n=1}{\overset{\infty}{\otimes}} C^*_Q(G_n) \cong
\underset{n=1}{\overset{\infty}{\otimes}} \mathcal{O}_{1+|\chi(G_n)|}.
\end{equation}
\end{thm}

{\em Acknowledgements.} Most of the work on this paper was done when I was a graduate student at Texas A$\&$M
University. 
I would like to thank Ron Douglas for choosing this research topic and for the 
many useful conversations I had with him. 
I want to thank Marcelo Laca for sending me hardcopies of some of his
papers on the subject. 
Finally I want to thank Ken Dykema, Marcelo Laca and George Elliott for some discussions.

\end{document}